\documentclass[a4paper,12pt]{amsart}
\usepackage[
  hmarginratio={1:1},     
  vmarginratio={1:1},     
  textwidth=160mm,        
  textheight=235mm,       
  marginparwidth=27mm,
  marginparsep=3mm
]{geometry}

\usepackage[utf8]{inputenc}
\usepackage{lmodern}
\usepackage{amsthm, amssymb, amscd, accents, bm}
\usepackage{mathtools}
\usepackage{mdwlist}
\usepackage{paralist}
\usepackage{graphicx}
\usepackage{subfigure}
\usepackage{epstopdf}
\usepackage{datetime}
\usepackage{afterpage}
\usepackage[hidelinks]{hyperref}
\usepackage{caption}
\usepackage{float}
\mathtoolsset{showonlyrefs}
\usepackage[sort,nocompress]{cite}
\DeclareMathOperator*{\argmin}{\mathrm{arg\, min}}

\theoremstyle{plain}
\newtheorem{definition}{Definition}[section]
\newtheorem{theorem}[definition]{Theorem}
\newtheorem{lemma}[definition]{Lemma}
\newtheorem{corollary}[definition]{Corollary}
\newtheorem{proposition}[definition]{Proposition}

\theoremstyle{definition}
\newtheorem{remark}[definition]{Remark}

\numberwithin{equation}{section}

\def\N{{\mathbb N}}
\def\Z{{\mathbb Z}}
\def\R{{\mathbb R}}
\def\C{{\mathbb C}}

\def\H{{\mathcal H}}
\newcommand{\toral}{{\ell^2_\lambda(\Z)}}
\newcommand{\foldedRange}[1][\lambda]{R_{#1,X}}
\newcommand{\dist}{\mathrm{dist}}
\newcommand{\id}{{\mathrm{id}}}
\newcommand{\sinc}{{\mathrm{sinc}}}
\newcommand{\ltZZ}{\ell^2(\Z,\Z)}
\newcommand{\ltZZi}{\ell^2(\Z,\Z[i])}
\newcommand{\re}{{\mathrm{Re} \,}}

\newcommand{\F}{{\mathcal{F}}}

\newcommand{\norm}[1]{\lVert #1\rVert}

\begin{document}
\title[Stability in unlimited sampling]{Stability in unlimited sampling}

\author[J. L. Romero]{Jos\'e Luis Romero}
\address[J. L. Romero]{Faculty of Mathematics, University of Vienna, Oskar-Morgenstern-Platz 1, A-1090 Vienna, Austria, and Acoustics Research Institute, Austrian Academy of Sciences, Dr. Ignaz Seipel-Platz 2,	AT-1010 Vienna, Austria}
\email[J. L. Romero]{jose.luis.romero@univie.ac.at}

\author[I. Shafkulovska]{Irina Shafkulovska}
\address[I. Shafkulovska]{Faculty of Mathematics, University of Vienna, Oskar-Morgenstern-Platz 1, A-1090 Vienna, Austria}
\email{irina.shafkulovska@univie.ac.at}

\thanks{J. L. R. and I. S. were funded in part or in whole by the Austrian Science Fund (FWF)  [\href{https://doi.org/10.55776/Y1199}{10.55776/Y1199}]. For open access purposes, the authors have applied a CC BY public copyright license to any author-accepted manuscript version arising from this submission.}
\allowdisplaybreaks
\belowdisplayshortskip0pt



\begin{abstract}
Folded sampling replaces clipping in analog-to-digital converters by reducing samples modulo a threshold, thereby avoiding saturation artifacts. We study the reconstruction of bandlimited functions from folded samples and show that, for equispaced sampling patterns, the recovery problem is inherently unstable. We then prove that imposing any a priori energy bound restores stability, and that this regularization effect extends to non-uniform sampling geometries. Our analysis recasts folded-sampling stability as an infinite-dimensional lattice shortest-vector problem, which we resolve via harmonic-analytic tools (the spectral profile of Fourier concentration matrices) and, alternatively, via bounds for integer Tschebyschev polynomials. Our work brings context to recent results on injectivity and encoding guarantees for folded sampling and further supports the empirical success of folded sampling under natural energy constraints.
\end{abstract}

\maketitle

\section{Introduction}
\subsection{Sampling rates, stability and folding}

The sampling theory for bandlimited functions links the \emph{bandwidth} of a class of signals to its \emph{Nyquist rate}, that is, the sampling rate at which signals can be faithfully encoded by recording their values at discrete locations. Bandwidth is mathematically described by means of the Fourier transform
$\hat{f}(\xi) = \int_\mathbb{R} f(t) e^{-2\pi i t \xi} \,dt\,$: A finite-energy signal is \emph{bandlimited}, with bandwidth $\Omega$, if it belongs to the \emph{Paley-Wiener space} 
\begin{equation}\label{eq:SWK}
    PW_{\Omega}(\R) = \lbrace f\in L^2(\R): \hat{f}=0 \text{ a.e. outside of }[-\tfrac{\Omega}{2}, \tfrac{\Omega}{2}]\rbrace.
\end{equation}
Let us recapture the basic facts of sampling theory in terms of the \emph{sampling operator} $C_X$ associated with an enumerated set $X = \{x_k: k \in \mathbb{Z}\} \subset \mathbb{R}$ of sampling locations:
\begin{align}
C_X: PW_{\Omega}(\R) \to \ell^2(\mathbb{Z}),
\qquad f \mapsto C_X f = (f(x_k))_{k\in \mathbb{Z}}.
\end{align}
Bandlimited functions $f \in PW_{\Omega}(\R)$ are completely determined by their samples $f(x_k)$, $k \in \mathbb{Z}$, exactly if $C_X$ is injective. When sampling at equispaced points $X= \alpha \mathbb{Z}$, the Nyquist-Shannon theorem \cite{Unser2000} shows that $C_{\alpha \mathbb{Z}}$ is injective if and only if $0<|\alpha| \leq 1/\Omega$. Moreover, in this case, the sampling operator $C_{\alpha \mathbb{Z}}$ admits a \emph{continuous left-inverse}, that is, a continuous map $D_X: \mathrm{Range}(C_{\alpha \mathbb{Z}}) \to PW_{\Omega}(\R)$ such that 
\begin{align}\label{eq_intro_cd}
D_X C_X f = f, \qquad f \in PW_{\Omega}(\R).
\end{align}
(Here, $\mathrm{Range}(C_{\alpha \mathbb{Z}})$ is the image or range of the sampling operator.) We can think of $D_X$ as a reconstruction operator that decodes the samples of a signal. In practical terms, the existence and continuity of the map $D_X$ means that signal samples can be reliably quantized, because similar collections of samples must correspond to similar signals. In this case, we speak of \emph{stable sampling}.

The distinction between injectivity and stability is more important for non-uniform sampling patterns. For example, the 
sampling operator $C_X$ associated with the set 
\begin{align}\label{eq_intro_X}
X=\{ 0.45 k : k \in \mathbb{N}\}
\end{align}
is injective but does not admit a continuous left inverse as in \eqref{eq_intro_cd}. In fact, sampling stability is related to approximate uniformity and equidistribution of sampling locations. Indeed, the celebrated sampling theorems of Duffin-Schaeffer and Beurling almost characterize the stability of a sampling pattern $X=\{x_k: k \in \mathbb{Z}\}$ in terms of its \emph{lower Beurling density}:
\begin{align}
D^{-}(X) = \liminf\limits_{r\to\infty}\inf\limits_{t\in\R} \tfrac{\#X\cap [t-r,t+r]}{2r}.
\end{align}
Assuming that there is a minimum separation between distinct points of $X$, the sampling operator $C_X$ admits a continuous left-inverse, as in \eqref{eq_intro_cd}, if its density exceeds the signals' bandwidth:
\begin{align}
D^{-}(X) > \Omega,
\end{align}
and only if $D^{-}(X) \geq \Omega$ (Landau's benchmark \cite{Landau1967, Landau1967IEEE}). On the other hand, the mere injectivity of $C_X$ is possible with very sparse sampling patterns such as \eqref{eq_intro_X}. In this case, the sampling encoding is fragile and not compatible with quantization.

The goal of this article is to revisit the topic of injectivity and stability of sampling in the context of a modern (non-linear) sampling method called \emph{folded sampling} \cite{BhandariEtAl2017}. This framework is motivated by practical restrictions on sampling architectures and, specifically, by the need to account for the maximal \emph{amplitude} $\lambda$ that recorded sampling values can assume in practice. When faced with samples having a high dynamic range, analog-to-digital converters (ADC) can easily reach their saturation threshold $\lambda$ and be forced to clip the numerical value: \[f(x) \mapsto \mathrm{sgn}(f(x)) \cdot \min\{|f(x)|, \lambda\},\]
which adds over-bandwidth frequencies and results in significant artifacts \cite{LinEtAl2025Seismic,DiezGarciaCamps2019, Gray1990,Ballou2015,MyszkowskiEtAl2024}.

As an alternative to clipping, Bhandari, Krahmer, and Raskar \cite{BhandariEtAl2017} proposed reacting to ADC saturation with a \emph{folding operation} that also maps arbitrary samples values into the admissible range $[-\lambda,\lambda]$. Precisely, the \emph{centered $\lambda$-folding} of a real number $x$ is the unique number $\mathcal{M}_\lambda(x) \in [-\lambda,\lambda)$ such that 
\begin{align}\label{eq_intro_mod}
x = \mathcal{M}_\lambda(x) + 2 \lambda n, \qquad \text{for some }n \in \mathbb{Z};
\end{align}
see Figure \ref{fig_1}.
In terms of the usual integer floor $\lfloor x \rfloor$ and decimal part $\{x\}= x-\lfloor x\rfloor$, the folding is implemented as 
\begin{equation}\label{eq_intro_fold2}
    \mathcal{M}_\lambda(x) = 2\lambda \left( \left\{ \tfrac{x}{2\lambda}+\tfrac{1}{2}\right\} -\tfrac{1}{2}\right).
\end{equation}
One also extends folding to complex numbers \cite{FernandezMenduinaEtAl2022}
by $\mathcal{M}_\lambda (x+iy) = \mathcal{M}_\lambda x + i \mathcal{M}_\lambda y$.
\begin{figure}[h!b]
	\subfigure{
		\includegraphics[width=.3\textwidth]{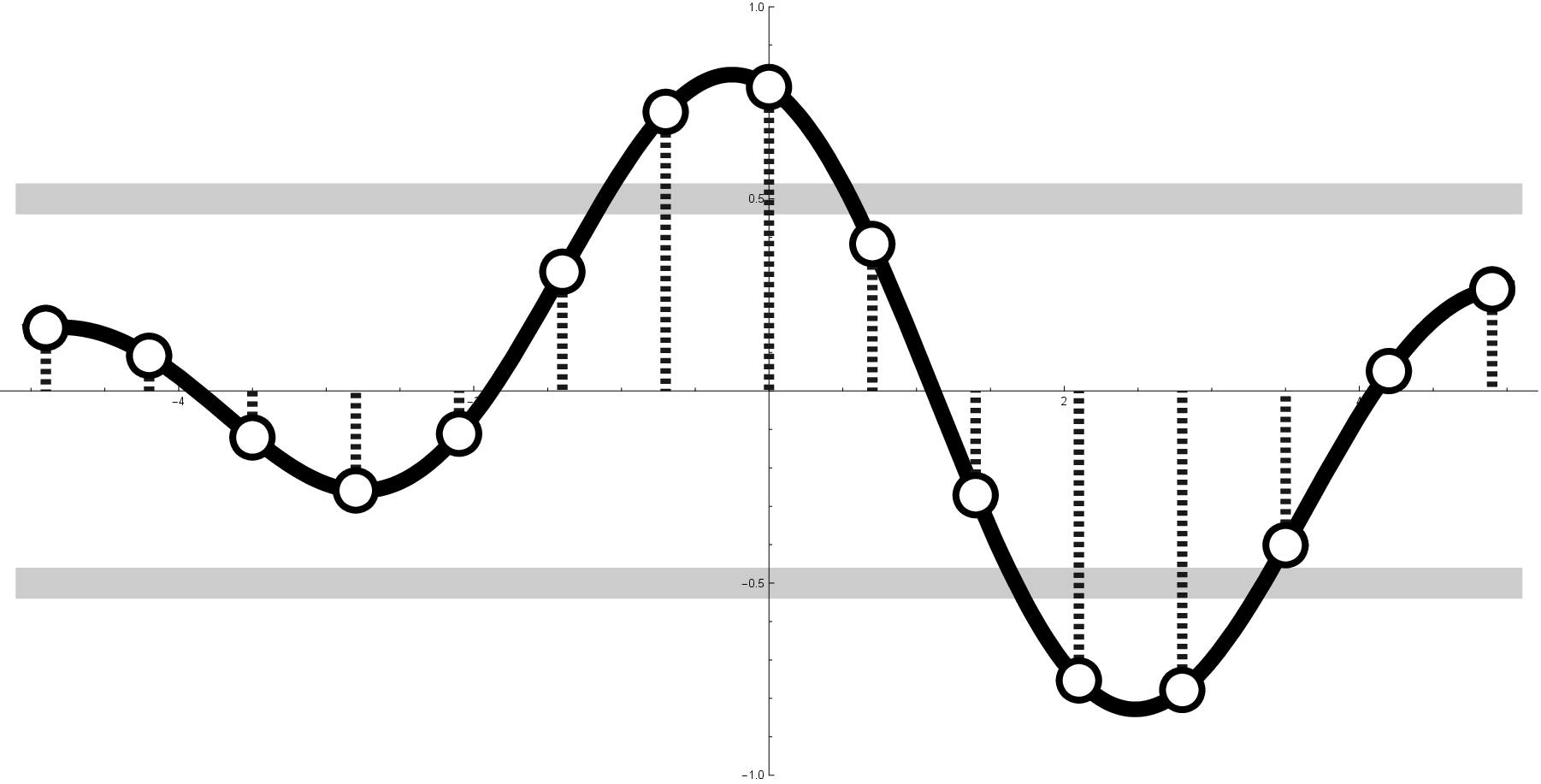}
	}
	\hfill
	\subfigure{
		\includegraphics[width=.3\textwidth]{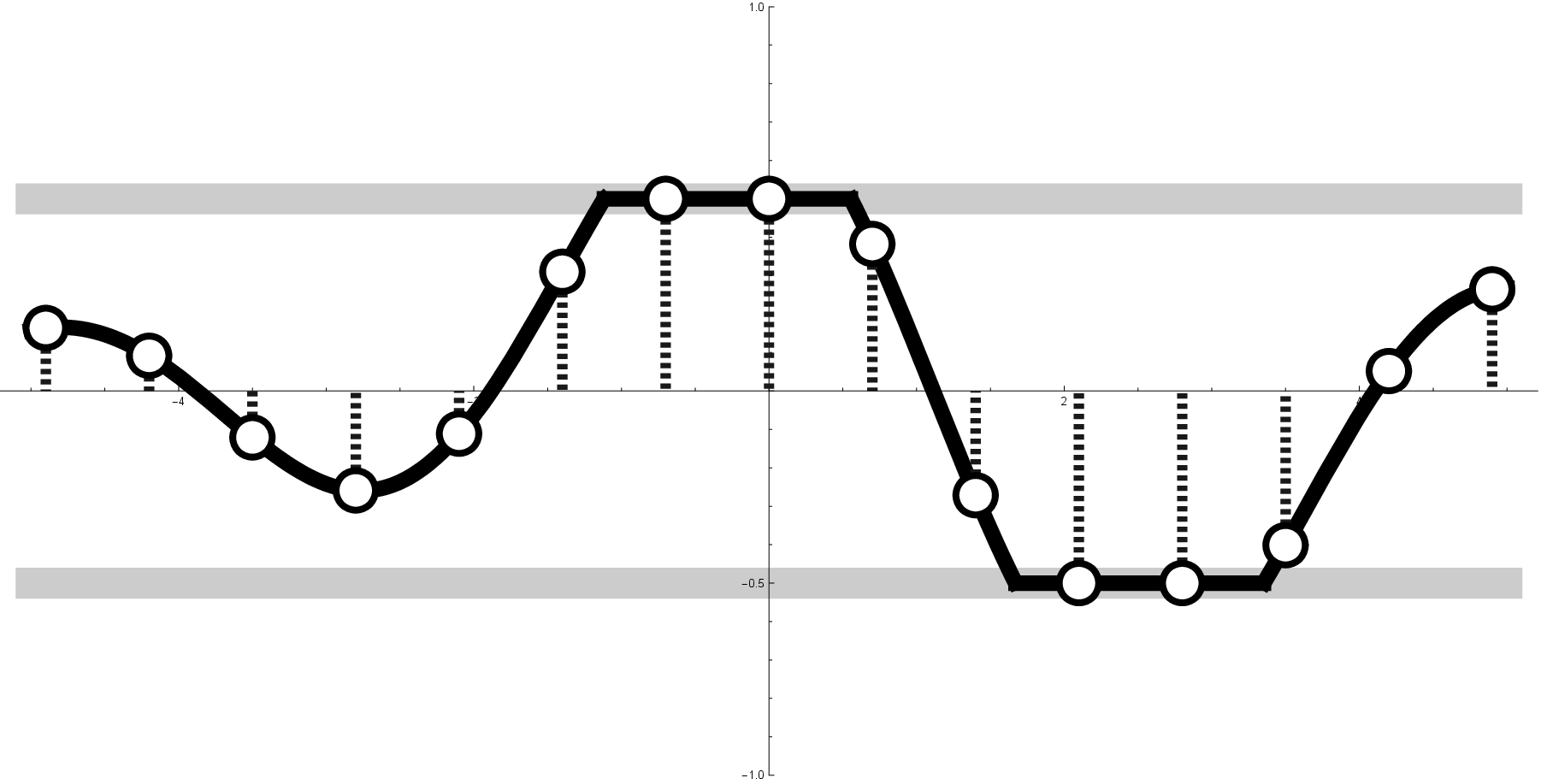}
	}
    \hfill
    \subfigure{		\includegraphics[width=.3\textwidth]{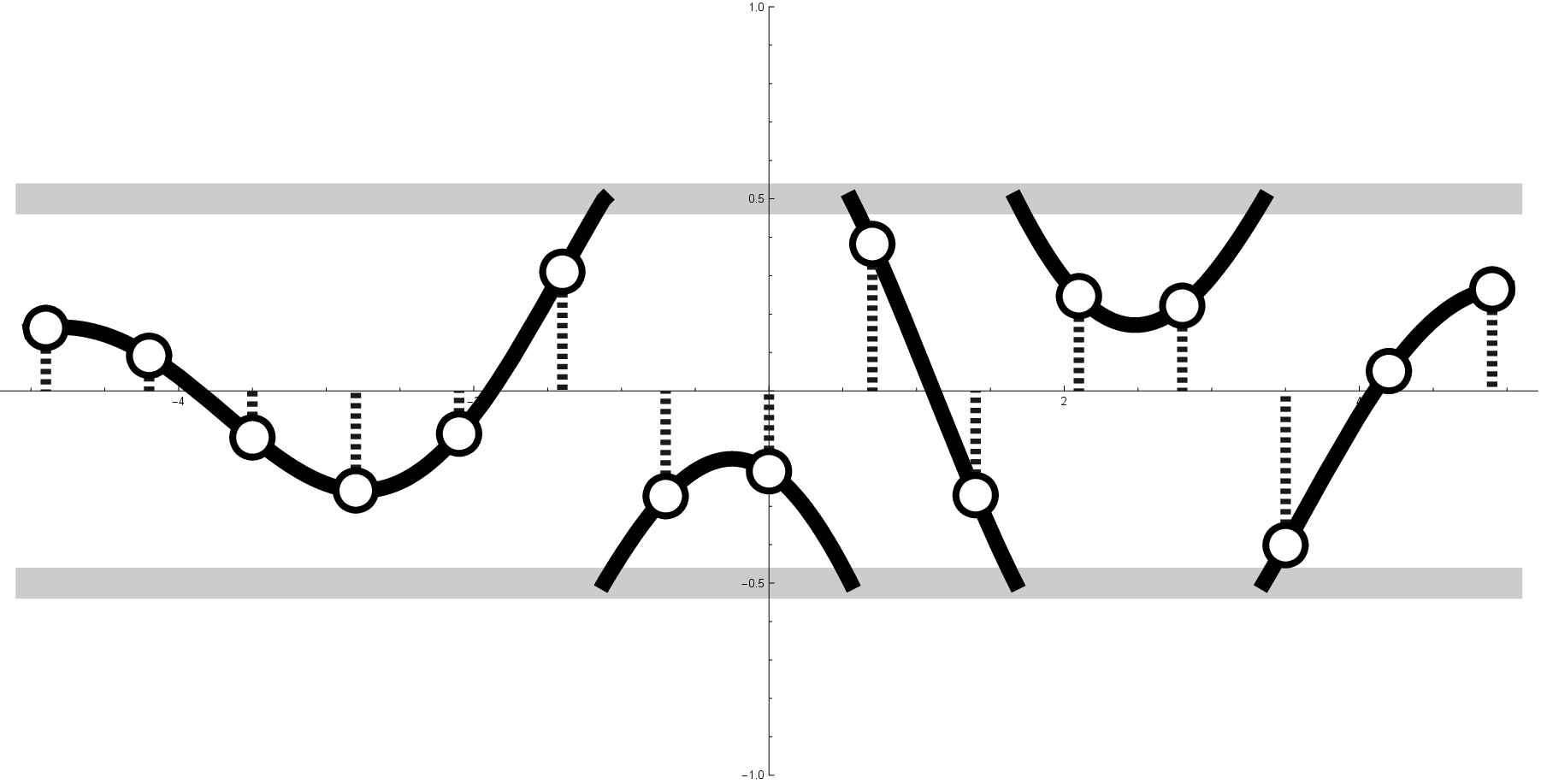}
    }
\caption{(a) Ideal sampling (expensive hardware); (b) clipped samples due to ADC saturation; (c) folding architecture - for dense sampling patterns, one can guess how to piece back the original functions, but less redundant sampling presents a challenge.}\label{fig_1}
\end{figure}
This sampling and folding framework was dubbed \emph{unlimited sensing} (US) \cite{BhandariEtAl2017} because it aims to avoid saturation-type limitations without sacrificing bits of precision. A great number of recent research has been devoted to validating the success of US, including theoretical recovery guarantees \cite{BhandariKrahmer2019,RomanovOrdentlich2019}, 
reconstruction algorithms \cite{BhandariEtAl2021_reconstruction, FlorescuEtAl2021,FlorescuBhandari2022,BhandariKrahmerPoskitt2022} 
and extensions to various novel settings \cite{AlharbiEtAl2025, AbdallaEtAl2025,RomanovOrdentlich2021,ZhangBhandari2023,rudresh2018wavelet_unlimited_sampling}. In parallel, US has been implemented in hardware \cite{BhandariKrahmerPoskitt2022}, with numerous continuous improvements \cite{ZhuBhandari2024,Mulleti2023HardwarePrototypeHDRADC,zhu2025ironing}.

In order to succeed, unlimited sensing requires a certain level of sampling redundancy. Indeed, the cardinal sine, $\sinc(x)=\tfrac{\sin(\pi x)}{\pi x}$, 
has bandwidth $\Omega=1$ - that is, $\sinc \in PW_1(\mathbb{R})$ - but takes integer values on the grid $X=\mathbb{Z}$. Thus, on this grid of critical density (Nyquist rate), folded sampling with $\lambda=1/2$ cannot distinguish $\sinc$ from the zero function. On the other hand, based on substantial numerical experience, practitioners expect that a moderate amount of sampling redundancy (over-Nyquist rate) provides enough extra information to effectively encode all bandlimited signals, even after discarding the integer component $n$ in \eqref{eq_intro_mod}. This was rigorously confirmed in \cite{BhandariKrahmer2019} and \cite{RomanovOrdentlich2019}, by proving the following:
\begin{theorem} \label{thm_injectivity}
Let $\Omega, \lambda >0$ and $0<\alpha<1/\Omega$. Then
the folded sampling operator
\begin{align}\label{eq_intro_fo}
\mathcal{M}_{\lambda,\alpha\Z}: PW_\Omega(\mathbb{R}) \to \ell^2(\mathbb{Z}), \qquad
f \mapsto \mathcal{M}_{\lambda,\alpha\Z} f = (\mathcal{M}_\lambda f(\alpha k))_{k\in\Z}
\end{align}
is injective.
\end{theorem}
In the words of Romanov and Ordentlich \cite{RomanovOrdentlich2019}, "Above the Nyquist Rate, Modulo Folding Does Not Hurt". Remarkably, the encoding guarantees in Theorem \ref{thm_injectivity} are \emph{independent} of $\lambda$, which means that their theoretical precision is independent of the available bit budget.

\section{Results}\label{sec_res}

In this article, we revisit Theorem \ref{thm_injectivity} from the point of view of stability. One first objection to this endeavor is that the folded sampling operator \eqref{eq_intro_fo} is not continuous, as the folding operation 
\eqref{eq_intro_fold2} has a discontinuity at every integer multiple of $\lambda$, see Section \ref{sec:new-crit}. We compensate for this by considering the \emph{toral (semi) metric}
\begin{align}\label{eq_toral_metric}
|z-w|_\lambda = \inf\limits_{n\in\Z+i\Z}|z-w-2\lambda n|, \qquad z,w \in \mathbb{C},
\end{align}
which makes the folding \eqref{eq_intro_fold2} continuous. Similarly, we consider
the space of square-summable sequences in the allowed folding range
\begin{align}
\toral = \{a \in \ell^2(\mathbb{Z}), a_k \in [-\lambda,\lambda)+ i [-\lambda,\lambda)\}
\end{align}
and endow it with the distance
\begin{align}
\|a-b\|^2_{\toral} = \sum_{k\in\mathbb{Z}} |a_k-b_k|_\lambda^2,
\end{align}
so that the folded sampling operator $C_X: PW_\Omega(\mathbb{R}) \to \ell_\lambda^2(\mathbb{Z})$ is continuous --- see Section \ref{sec:new-crit}. Our first result reads as follows.
\begin{theorem}[Instability of folded sampling with unbounded inputs]\label{thm:main}
Let $\Omega, \lambda, \alpha >0$. Then the folded sampling operator
  \begin{equation}\label{eq_intro_fso}
     \mathcal{M}_{\lambda,\alpha\Z}: PW_\Omega(\mathbb{R}) \to \toral, \qquad
f \mapsto \mathcal{M}_{\lambda,\alpha\Z} f = (\mathcal{M}_\lambda f(\alpha k))_{k\in\Z}
  \end{equation}
  does not admit a left-inverse
  $\mathcal{T}_{\lambda, \alpha\Z}: \mathrm{Range}(\mathcal{M}_{\lambda,\alpha\Z}) \to PW_\Omega(\mathbb{R})$
  which is continuous anywhere. More concretely,
      for all $f\in PW_\Omega(\R)$ there exists an $\varepsilon>0$ with the following property: for all $\delta>0$ there exists a function $g\in PW_\Omega(\R)$ such that
      \begin{equation}
          \norm{\mathcal{M}_{\lambda,X} f - \mathcal{M}_{\lambda,X} g}_{\toral}<\delta,\text{\quad but \quad}
          \norm{f- g}_{L^2(\R)}>\varepsilon.
      \end{equation}
      The statement remains true even if we restrict the operator to real-valued functions.
\end{theorem}
Theorem \ref{thm:main} implies that the encoding guarantee of Theorem \ref{thm_injectivity} is fragile in the sense that no signal $f$ can be encoded with arbitrary precision, no matter how many bits are invested in recording its modulo samples.

While Theorem \ref{thm:main} brings some context to \cite{BhandariKrahmer2019} and \cite{RomanovOrdentlich2019}, it certainly does not negate their main conclusion that folded sampling is a very effective encoder. In fact, the recovery guarantees in 
\cite{YanEtAl2025,BhandariEtAl2021_reconstruction,BhandariKrahmer2019,RomanovOrdentlich2019} are proven by constructive means that imply a certain level of robustness and stability, albeit for \emph{a priori bounded inputs}. As a second contribution, we investigate this phenomenon in general and prove the following.

\begin{theorem}\label{thm:main_bounded}
Let $\Omega, \lambda>0$, and let $X=\{x_k:k\in\mathbb{Z}\}\subseteq \R$ be a uniformly separated set with lower Beurling density
\begin{equation}\label{eq:def:lower_beurling}
    D^-(X) = \liminf\limits_{r\to\infty}\inf\limits_{t\in\R} \tfrac{\#X\cap [t-r,t+r]}{2r}>\Omega.
\end{equation}
Let $\mathcal{B} \subseteq PW_\Omega(\R)$ be a bounded set. Then the restriction of the folded sampling operator
\begin{equation}\label{eq_intro_fso_B}
     \mathcal{M}_{\lambda,X}: PW_\Omega(\mathbb{R}) \to \toral, \qquad
f \mapsto \mathcal{M}_{\lambda,X} f = (\mathcal{M}_\lambda f(x_k))_{k\in\Z}
\end{equation}
to $\mathcal{B}$ has a Lipschitz continuous left inverse on its range.\ That is, there exists a positive constant $C=C_{\mathcal{B},\lambda,X}>0$ such that:
    \begin{equation}
        \norm{f-g}_{L^2(\R)} \leq C \norm{\mathcal{M}_{\lambda,X} f -\mathcal{M}_{\lambda,X} g}_{\toral}, \qquad
        f,g\in \mathcal{B}.
    \end{equation}    
\end{theorem}
The term \emph{uniformly separated} in Theorem \ref{thm:main_bounded} means that $\inf_{x,y \in X, x\not=y}|x-y|>0$. For uniform sampling patterns $X=\alpha \mathbb{Z}$, similar stability guarantees can be proved by following the constructive arguments of \cite{YanEtAl2025,BhandariEtAl2021_reconstruction,BhandariKrahmer2019,RomanovOrdentlich2019}, based on divided differences and linear prediction theory. For general sampling geometries, we shall resort to less constructive arguments --- see Section \ref{sec_tech}. As a corollary
of Theorem \ref{thm:main_bounded}, we obtain an extension of the injectivity theorem of \cite{BhandariKrahmer2019,RomanovOrdentlich2019} to general acquisition geometries.
\begin{corollary}\label{thm:main_inj_X}
Let $\Omega, \lambda>0$, and let $X\subseteq \R$ be a uniformly separated set with $D^-(X) >\Omega$. Then the folded sampling operator \eqref{eq_intro_fso} is injective.
\end{corollary}

\subsection{Technical overview}\label{sec_tech}
Our main insight is that the stability of the folded sampling operator is equivalent to a lattice distance problem concerning the range of the non-folded sampling operator $C_X$. To exploit this insight, we bring to bear the spectral properties of \emph{prolate matrices}, which are discrete versions of the time-band limiting operators studied by Landau, Pollak and Slepian \cite{LandauPollak1961,LandauPollak1962,SlepianPollak1961,Slepian1978}. Theorem \ref{thm:main} is then obtained by studying a certain shortest vector problem (SVP). We also present an alternative, more constructive approach 
to Theorem \ref{thm:main} based on \emph{Tschebyschev polynomials with integer coefficients}, which are a variant of the classical orthogonal polynomials with the additional restriction of having integer coefficients. Interestingly, classical Tschebyschev polynomials also play a role in \cite{RomanovOrdentlich2019}. We hope that the two arguments for Theorem \ref{thm:main} provide useful complementary insights. For a special range of $\alpha$ in Theorem \ref{thm:main}, we are able to provide a fully constructive proof. 

As for the proof of Theorem \ref{thm:main_bounded}, we 
elaborate on the insight of \cite{BhandariKrahmer2019}, which compares folded and unfolded sampling by discarding large samples. While in the context of uniform sampling, this is analyzed by directly inspecting the Shannon sampling formula \cite{BhandariKrahmer2019}, in our case, we invoke Beurling's balayage theory \cite{BeurlingCollected1989}.

\subsection{Outline}
In Section \ref{sec:frame_theory}, we recall the necessary tools from frame theory and sampling in Paley-Wiener spaces. In Section \ref{sec:new-crit}, we inspect the continuity of the folded sampling operator and provide 
a criterion for its continuous left-invertibility in terms of a certain distance functional. Section \ref{sec:lattice_tools} analyzes the distance function in the case of lattice sampling geometries and makes the general criterion more concrete. With this preparation, the first proof of Theorem \ref{thm:main} is presented in Section \ref{sec:main_via_SVP}. In Section \ref{sec:main_via_TP}, we provide an additional completely constructive argument for Theorem \ref{thm:main}, valid albeit in the range $\alpha \Omega > {\frac{2}{3}}$. This is used as motivation for a second proof of Theorem \ref{thm:main}.
Section \ref{sec:bounded_case} contains a proof of Theorem \ref{thm:main_bounded}. Section \ref{sec_con} presents a brief conclusion, while a certain technical result about sampling theory is postponed to Section \ref{proof_samp}.

\section{Frame theory and sampling essentials}\label{sec:frame_theory}
We recall the elementary facts of frame theory and refer to Christensen's book \cite{Christensen2016} for proofs and further details. 
For clarity, we introduce the terminology in the context of Hilbert spaces, though we are only interested in the Paley-Wiener space, which is the motivating example \cite{DuffinSchaeffer1952}.

Let $\H$ be a separable Hilbert space. A sequence of elements $\Psi = (\psi_k)_{k\in\Z}\subseteq\H$ is called a frame if there exist constants $0<A\leq B<\infty$ such that 
\begin{equation}\label{eq_frame_1}
    A\norm{f}_\H^2 \leq \norm{\left(\langle f, \psi_k\rangle_\H \right)_{k\in\Z}}_{\ell^2(\Z)}^2 \leq B \norm{f}_\H^2,\quad f\in\H.
\end{equation}
The optimal $A_\Psi$ and $B_\Psi$ in \eqref{eq_frame_1} are called the lower and the upper frame bound, respectively.
Let us fix a frame $\Psi$ with frame bounds $A_\Psi$ and $B_\Psi$. The \emph{analysis operator} of $\Psi$ is the linear mapping
\begin{equation}
    C_\Psi: \H \to \ell^2(\Z), \quad f\mapsto \left(\langle f, \psi_k\rangle_\H \right)_{k\in\Z}.
\end{equation}
This is a bounded linear operator, with operator norm $\norm{C_\Psi}_{\H\to\ell^2(\Z)} = B_\Psi^{1/2}$. Its adjoint is the \emph{synthesis operator} 
\begin{equation}
    C_\Psi^*: \ell^2(\Z) \to \H, \quad c\mapsto \sum\limits_{k\in\Z} c_k \psi_k.
\end{equation}
Their composition $S_\Psi =C_\Psi^*C_\Psi$ is called the \emph{frame operator}. This is a positive definite operator, and its spectrum is contained in $[A_\Psi,B_\Psi]$. 

The reverse composition $G_\Psi = C_\Psi C_\Psi^*$ is called the \emph{Gram matrix}. It is a bounded operator on $\ell^2(\mathbb{Z})$ that we identify with the bi-infinite matrix
\begin{equation}
    (G_\Psi)_{j,k} = \langle \psi_k,\psi_j\rangle_\H, \quad j,k\in\Z.
\end{equation}
We denote with $R_\Psi$ the range of $C_\Psi$.
This is a closed subspace of $\ell^2(\Z)$. 

The analysis operator is not necessarily surjective, but it is injective due to \eqref{eq_frame_1}. It has a bounded (linear) inverse on its range, and satisfies
$\norm{C_\Psi^{-1}}_{R_\Psi\to \H}=A^{-1/2}_\Psi$. 

The orthogonal projection onto $R_\Psi$ is given by
\begin{equation}\label{eq:P_psi_proj}
    P_\Psi = C_\Psi S_\Psi^{-1}C_\Psi^* = C_\Psi (C_\Psi^*C_\Psi)^{-1}C_\Psi^* = G_\Psi G_\Psi^\dagger = G_\Psi^\dagger G_\Psi,
\end{equation}
where $G_\Psi^\dagger$ denotes the Moore-Penrose pseudoinverse of $G_\Psi$. 
The projection onto the orthogonal complement is given by $\id-P_\Psi$.

The Paley-Wiener space $PW_\Omega(\R)$ 
is a closed subspace of $L^2(\R)$ and thus a Hilbert space. It is the prototypical example of a \emph{reproducing kernel Hilbert space}, that is, a Hilbert space of functions such that the pointwise evaluation functionals are continuous. The \emph{reproducing kernel} of $PW_\Omega(\R)$ at the point $x \in \mathbb{R}$ is the function $k_x(y) = \Omega\,\sinc(\Omega(y-x))$, which has the property
\begin{align}
f(x) = \langle f,k_x\rangle, \qquad f \in PW_\Omega(\R).
\end{align}
For $\psi_k(y) = \Omega\,\sinc(\Omega(y-x_k))$, the frame property \eqref{eq_frame_1} is equivalent to the \emph{sampling inequality} 
\begin{equation}\label{eq:sampling_ineq_def}
    A\norm{f}_\H^2 \leq \norm{\left(f(x_k)\right)_{k\in\Z}}_{\ell^2(\Z)}^2 \leq B \norm{f}_\H^2,
    \quad PW_\Omega(\R),
\end{equation}
and we also refer to $C_\Psi$ as the \emph{sampling operator} of the sampling set $X=\{x_k:k\in\mathbb{Z}\}$. The corresponding Gram matrix is
\begin{equation}\label{eq:Gram_kernel_case}
    (G_\Psi)_{j,k} = \langle \psi_k,\psi_j\rangle_\H = \psi_{k}(x_j)
    =\Omega\,\sinc(\Omega(x_j-x_k)),\quad j,k\in\Z.
\end{equation}
When working with frames of reproducing kernels, we replace the index $\Psi$ with the sampling set $X =\{x_k:k\in\Z\}$ 
and use the sampling/frame terminology interchangeably.

\subsection{Sampling in the Paley-Wiener space}
The basic result about sampling bandlimited functions on lattices goes back at least to Whittaker \cite{Whittaker1915}, Shannon \cite{Shannon1948,shannon1949}, and Kotelnikov \cite{Kotelnikov2001}. In the language introduced above, we have that a lattice $\alpha\Z$, with $\alpha>0$, is a sampling set for $PW_\Omega(\mathbb{R})$ if and only if $0<\alpha\leq 1/\Omega$, in which case both frame bounds in \eqref{eq:sampling_ineq_def} equal $1/\alpha$. If $\alpha=1/\Omega$, then the the range of the sampling operator is $R_\Z = \ell^2(\Z)$ and the sampling operator $C_{\Z}$ is an isometric isomorphism between $PW_\Omega(\mathbb{R})$ and $\ell^2(\mathbb{Z})$. If $0<\alpha<1/\Omega$, this is not the case anymore and $R_{\alpha\Z}$ is a non-trivial closed subspace of $\ell^2(\Z)$, which we shall carefully inspect.

With respect to general sampling patterns, the most conclusive results are due to Beurling \cite{BeurlingCollected1989}, 
Duffin-Schaeffer \cite{DuffinSchaeffer1952} and
Landau \cite{Landau1967}, and are usually formulated in terms of the \emph{lower Beurling density} of a set $X=\{x_k: k \in \mathbb{Z}\}$:
\begin{align}\label{eq_bden}
D^{-}(X) = \liminf\limits_{r\to\infty}\inf\limits_{t\in\R} \tfrac{\#X\cap [t-r,t+r]}{2r}.
\end{align}
The density measures the asymptotic average of points per unit area, and, for a lattice, it is $D^-(\alpha\Z) = 1/\alpha$. In order for the count in \eqref{eq_bden} to be more meaningful, one often assumes that $X$ is \emph{uniformly separated}:
\begin{equation}
    \inf\{|x-x'|\,:\,x,x'\in X, \\ x\neq x'\}>0,
\end{equation}
as pairs of sampling locations that are arbitrarily close together add little distinctive information, but do increase \eqref{eq_bden}. Uniform separation also ensures that the sampling operator $C_X:PW_\Omega(\mathbb{R})\to\ell^2(\mathbb{Z})$ is bounded. 

Beurling's sampling theorem states that a uniformly separated set $X$ is a sampling set for $PW_\Omega(\mathbb{R})$ if $D^-(X)>\Omega$ and only if $D^-(X)\geq \Omega$ \cite{BeurlingCollected1989,Landau1967}, while the full characterization of the sampling property is more delicate \cite{MR1923965}. We will need a more precise version of Beurling's sampling theorem that shows which geometric qualities of $X$ impact the corresponding sampling (frame) bounds. The following result is undoubtedly known to experts, but, lacking a precise reference, we provide a proof of it in Section \ref{proof_samp}.

\begin{theorem}[Qualitative version of the sampling theorem]\label{thm:bound_dependency}
Let $\Omega$, $\delta>0$ and let $X=\{x_k:k\in\mathbb{Z}\} \subseteq \R$ be a uniformly separated set with
\begin{equation}
    \inf\limits_{j,k\in\Z, j\neq k}|x_j-x_k|\geq \delta >0.
\end{equation}
Let $\varepsilon>0$ and $r>0$ be such that 
\begin{equation}\label{eq:choice_r}
    \inf\limits_{x\in\R} \tfrac{\#X\cap [t-r,t+r]}{2r} \geq \Omega+\varepsilon.
\end{equation}
Then $X$ is a sampling set for $PW_\Omega(\R)$ and the lower frame bound $A_X$ can be bounded from below by a positive constant $A(\Omega,\delta, \varepsilon, r)>0$ that only depends on $\Omega$, $\delta$, $\varepsilon$, and $r$.
\end{theorem}

\section{Continuity as a distance problem}\label{sec:new-crit} 
\subsection{(Dis-)continuity of the forward mapping}
As a warm-up, we begin with a simple observation.
\begin{lemma}[Discontinuity of the folded sampling operator with the Euclidean metric]\label{lem_disc}
Let $\Omega,\lambda>0$, and let $X=\{x_k:k\in\mathbb{Z}\}\subseteq\R$ be a uniformly separated set. Then the map
\begin{equation}
    \mathcal{M}_{\lambda,X}: PW_{\Omega}(\R)\to\ell^2(\Z), \quad f\mapsto \big( \mathcal{M}_\lambda f(x_k)\big)_{k\in \Z}
\end{equation}
is discontinuous.
\end{lemma}
\begin{proof}
Let $f(x) = \lambda \tfrac{\sin(\pi \Omega (x-x_0))}{\pi \Omega (x-x_0)}$
and note that, for $1/2<\gamma<1$,
\begin{align}
\mathcal{M}_\lambda (\gamma f) = \gamma f \quad \text{and}\quad
\mathcal{M}_\lambda (\tfrac{1}{\gamma}f)(x_0) = \tfrac{f(x_0)}{\gamma}-2\lambda.
\end{align}
Thus 
\begin{equation}       \lim\limits_{\gamma\nearrow 1} \norm{(\mathcal{M}_{\lambda,X}(\gamma f) - \mathcal{M}_{\lambda, X}(\tfrac{1}{\gamma}f)) }_{\ell^2(\Z)}
     \geq \lim\limits_{\gamma\nearrow 1} |\gamma f(x_0) - \tfrac{1}{\gamma}f(x_0)+2\lambda| =2\lambda> 0, 
\end{equation}
while $
    \lim\limits_{\gamma\nearrow 1} \norm{\gamma f - \tfrac{1}{\gamma}f }_{L^2(\R)} =0$; see Figure \ref{fig:forward_unstable}.
\end{proof}
\begin{figure}[h]\includegraphics[width=.5\textwidth]{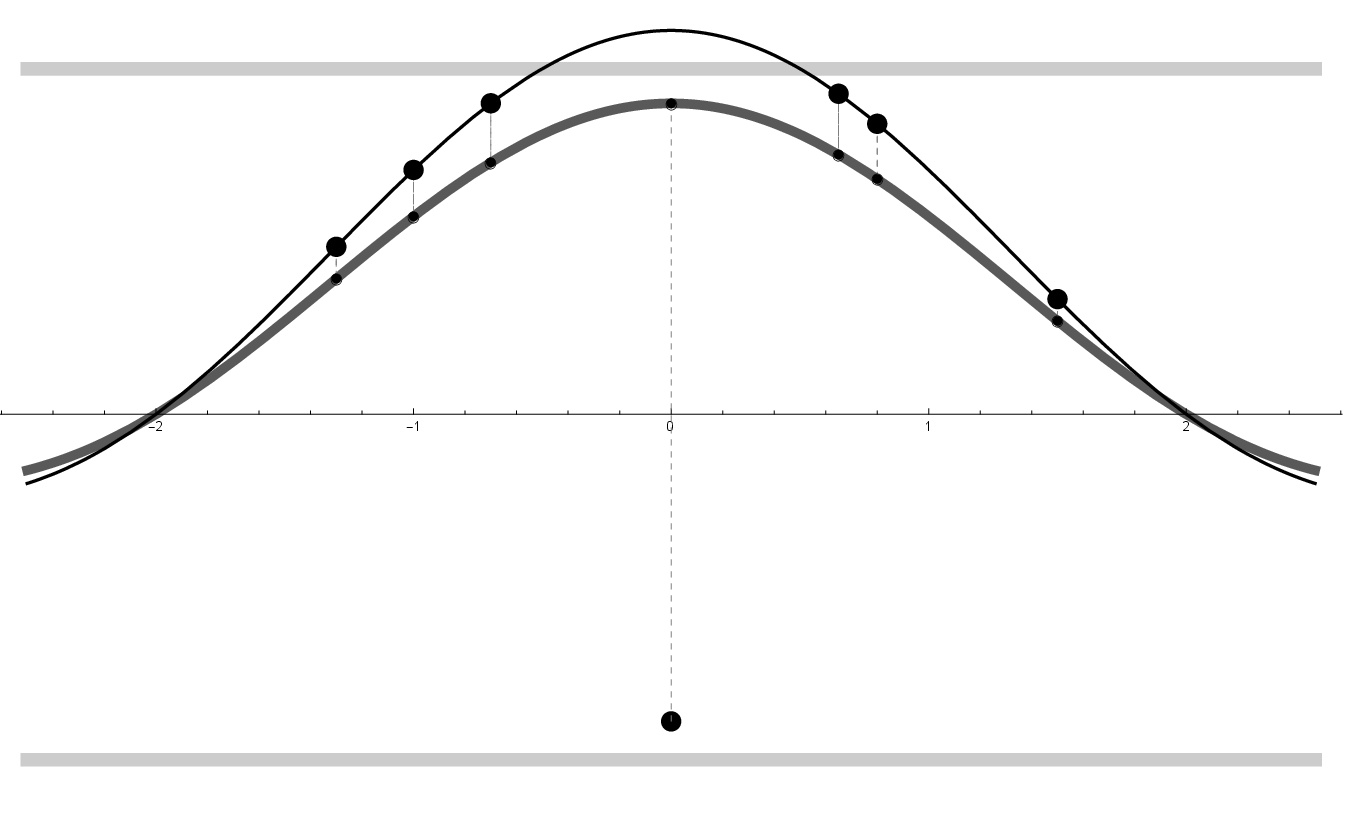}
	\caption{The functions $\frac{f}{\gamma}$ (black) and $\gamma f$ (gray) from Lemma \ref{lem_disc} sampled around $x_0$.}\label{fig:forward_unstable}
\end{figure} 
As explained in Section \ref{sec_res}, the toral metric \eqref{eq_toral_metric} is better suited to the analysis of the folding operation. Let us make a few useful observations.

To shorten the notation, we write $\Z[i] = \Z+i\Z$ for the Gaussian integers.
\begin{lemma}\label{lem:toral_translation_inv}
For all $\lambda>0$ the toral metric $(z,w) \mapsto |z-w|_\lambda$ is translation invariant, with
     \begin{equation}\label{eq_b3}
        |z-w|_\lambda =|\mathcal{M}_\lambda(z) - \mathcal{M}_{\lambda}(w)|_\lambda =|\mathcal{M}_\lambda(z-w)|_\lambda =|\mathcal{M}_\lambda(z-w)|,\quad z,w\in\C.
     \end{equation}
     In particular, for all sequences $a,b\in\ell^2(\Z)$:      \begin{equation}\label{eq:toral_transion_inv_norm}
         \left\lVert \mathcal{M}_\lambda(a) - \mathcal{M}_{\lambda}(b) \right\rVert_\toral = \left\lVert \mathcal{M}_\lambda(a-b) \right\rVert_{\toral} = \left\lVert \mathcal{M}_\lambda(a-b) \right\rVert_{\ell^2(\Z)}.
     \end{equation}
Secondly, for all $z,w \in \mathbb{C}$:
\begin{align}\label{eq_three}
\text{if}\quad|\mathcal{M}_\lambda(z)-\mathcal{M}_\lambda(w)|<\lambda,
\quad\mbox{then}\quad \mathcal{M}_\lambda(z)-\mathcal{M}_\lambda(w)=\mathcal{M}_\lambda(z-w).
\end{align}
\end{lemma}
 \begin{proof}
 The translation invariance is obvious. Furthermore, for all $m\in\Z[i]$:
     \begin{equation}
         |z+2\lambda m|_\lambda = \inf\limits_{n\in\Z[i]}|z+2\lambda m+2\lambda n| = \inf\limits_{n\in\Z[i]}|z+2\lambda n| = |z|_\lambda = |\mathcal{M}_\lambda (z)|.
     \end{equation}
     As $z-w$, $\mathcal{M}_\lambda (z-w)$, $\mathcal{M}_\lambda(z) - \mathcal{M}_\lambda(w)$ are all congruent modulo $2\lambda\mathbb{Z}[i]$, \eqref{eq_b3} follows.

     Similarly, if $|\mathcal{M}_\lambda(z)-\mathcal{M}_\lambda(w)|<\lambda$, then $\mathcal{M}_\lambda(z)-\mathcal{M}_\lambda(w) \in (-\lambda,\lambda)+i(-\lambda,\lambda)$. Since $z-w$ and $\mathcal{M}_\lambda(z)-\mathcal{M}_\lambda(w)$ are congruent modulo $2\lambda\mathbb{Z}[i]$, it follows that $\mathcal{M}_\lambda(z-w)=\mathcal{M}_\lambda(z)-\mathcal{M}_\lambda(w)$.
 \end{proof}
We are now in a position to show the continuity of the folded sampling operator with respect to the toral metric.
\begin{lemma}\label{lem_fol_cont}
Let $\Omega, \lambda>0$, and let $X=\{x_k:k\in\mathbb{Z}\}\subseteq \R$ be a uniformly separated set. Then the folded sampling operator
\begin{equation}
    \mathcal{M}_{\lambda,X}: PW_\Omega(\mathbb{R}) \to \ell_\lambda^2(\mathbb{Z}), \qquad
f \mapsto \mathcal{M}_{\lambda,X} f = (\mathcal{M}_\lambda f(x_k))_{k\in\Z}
  \end{equation}
is Lipschitz continuous.
\end{lemma}
\begin{proof}
Using Lemma \ref{lem:toral_translation_inv} we have, 
\begin{equation} 
  \begin{split}
        \phantom{\leq}\ \,\norm{\mathcal{M}_{\lambda,X}f-\mathcal{M}_{\lambda,X} g}_{\toral}
        &=\norm{\mathcal{M}_{\lambda,X}(f-g)}_{\toral}  
        \leq \norm{C_X(f-g)}_{\ell^2(\Z)}     \\ & \leq \norm{C_X}_{L^2(\R)\to \ell^2(\Z)}\norm{f-g}_{L^2(\R)},
  \end{split}
\end{equation}
while $\norm{C_X}_{L^2(\R)\to \ell^2(\Z)}<\infty$ because $X$ is uniformly separated.
\end{proof}

\subsection{Euclidean vs toral stability}
We now explore the interplay between the left-invertibility of the folded  sampling operator $\mathcal{M}_{\lambda,X}$, when its range 
\begin{align}
\foldedRange :=
\mathcal{M}_{\lambda,X} ( PW_\Omega(\mathbb{R}) )
\end{align}
is considered either with the metric inherited from $\ell^2(\mathbb{Z})$
(Euclidean) or $\toral$ (toral). 
 
\begin{lemma}[Reformulations of continuity]\label{lem:continuity_statements} 
    Let $\lambda>0$, and let $X = \{x_k\}_{k\in\Z}\subseteq\R$ be uniformly separated set. Assume that 
    there exists a map $$\mathcal{T}_{\lambda,X}: \foldedRange \to PW_{\Omega}(\R)$$
    that is a left-inverse of $\mathcal{M}_{\lambda,X}$, i.e., $\mathcal{T}_{\lambda,X} \mathcal{M}_{\lambda,X} = \id_{PW}$. Then the following hold.
    \begin{enumerate}[(i)]
        \item $\mathcal{T}_{\lambda,X}$ is uniformly continuous with respect to the Euclidean metric if and only if it is continuous with respect to the Euclidean metric at the zero sequence.
    \item $\mathcal{T}_{\lambda,X}$ is uniformly continuous with respect to the toral metric if and only if it is continuous with respect to the toral metric at some sequence $\mathcal{M}_{\lambda,X}h_0$.
    \item If $\mathcal{T}_{\lambda,X}$ is continuous {with respect to the toral metric} at a  sequence $\mathcal{M}_{\lambda,X}h_0$, then it is also continuous with respect to the Euclidean metric at the same sequence $\mathcal{M}_{\lambda,X}h_0$. 
    \end{enumerate}
\end{lemma}
\begin{proof}
\textbf{(i) Continuity w.r.t. Euclidean metric.} \label{lem:cont_euclidean} 
Suppose that $\mathcal{T}_{\lambda,X}$ is continuous at $0$ with the Euclidean metric and let $\varepsilon>0$. Then there exists $0<\delta<\lambda$ such that, for $h \in PW_\Omega(\mathbb{R})$,
\begin{equation}
    \norm{\mathcal{M}_{\lambda,X} h}_{\ell^2(\Z)}<\delta 
    \text{\quad implies \quad}
    \norm{h}_{L^2(\R)}<\varepsilon.
    \end{equation}
Let $f,g\in PW_{\Omega}(\R)$ such that 
     \begin{equation}
        \norm{\mathcal{M}_{\lambda,X} f - \mathcal{M}_{\lambda,X} g}_{\ell^2(\Z)}<
        \delta.
    \end{equation}
    Since $\norm{\mathcal{M}_{\lambda,X} f - \mathcal{M}_{\lambda,X} g}_{\ell^\infty(\Z)} \leq \norm{\mathcal{M}_{\lambda,X} f - \mathcal{M}_{\lambda,X} g}_{\ell^2(\Z)}$, by
    \eqref{eq_three} with $z = f(x_k), w=g(x_k)$, $k\in\Z$, the function $ h=f-g$ satisfies
    \begin{equation}
        \norm{\mathcal{M}_{\lambda,X} h}_{\ell^2(\Z)} = \norm{\mathcal{M}_{\lambda,X} f - \mathcal{M}_{\lambda,X} g}_{\ell^2(\Z)}<\delta.
    \end{equation}
    Thus, $\norm{f-g}_{L^2(\R)} = \norm{h}_{L^2(\R)} 
    <\varepsilon$, as desired.

    \medskip

    \noindent \textbf{(ii) Continuity w.r.t. toral metric.} \label{lem:cont_everywhere_toral} 
    Suppose that $\mathcal{T}_{\lambda,X}$ is continuous at $\mathcal{M}_{\lambda,X}h_0$ for some
    $h_0 \in PW_\Omega(\mathbb{R})$ and let $\varepsilon>0$. Then there exists $\delta>0$ such that     \begin{equation}\label{eq:toral_h_0_def}
        \norm{\mathcal{M}_{\lambda,X} h - \mathcal{M}_{\lambda,X} h_0}_{\toral}<\delta 
        \text{\quad implies \quad}
        \norm{h-h_0}_{L^2(\R)}<\varepsilon
        \end{equation}
        for $h\in PW_\Omega(\R)$.
        Let $f,g\in PW_\Omega(\R)$. By Lemma \ref{lem:toral_translation_inv},
        \begin{equation}
            \norm{\mathcal{M}_{\lambda,X} f - \mathcal{M}_{\lambda,X} g}_{\toral} = \norm{\mathcal{M}_{\lambda,X} (f-g)}_{\toral} = \norm{\mathcal{M}_{\lambda,X} (f-g+h_0) - \mathcal{M}_{\lambda,X} h_0}_{\toral}.
        \end{equation}
        Thus, if $          \norm{\mathcal{M}_{\lambda,X} f - \mathcal{M}_{\lambda,X} g}_{\toral}<\delta$,
        then, by \eqref{eq:toral_h_0_def},
        \begin{equation}
        \norm{f-g}_{L^2(\R)} = \norm{(f-g+h_0)-h_0}_{L^2(\R)} <\varepsilon.
        \end{equation}

 \noindent\textbf{(iii) Comparison of topologies.} Since $|z-w|_\lambda \leq |z-w|$ for all $z,w\in\C$, the toral metric on $\foldedRange$ defines a coarser topology than the one induced by the Euclidean metric, and the claim follows.
\end{proof}

\subsection{The distance criterion}
We now come to one of our main insights: a criterion for the continuous left invertibility of the folded sampling map in terms of a certain integer distance problem. Let us introduce some notation.
\begin{definition} 
    Let $X\subseteq\R$ be a sampling set for $PW_1(\R)$. We denote with $\ltZZi$ and $\ltZZ$ the square-summable 
    (in fact, finitely supported) bi-infinite sequences with entries in $\Z[i] = \Z+i\Z$ and $\Z$ respectively: 
\begin{align*}
    \ltZZi &=\{ n\in\ell^2(\Z):\  n_k\in \Z[i],\, k\in\Z\},\\
           \ltZZ &=\{ n\in\ell^2(\Z):\  n_k\in \Z,\, k\in\Z\}.
\end{align*}
Recall that $R_X$ is the range of the sampling operator $C_X$, and $\id -P_X$ is the projection onto its orthogonal complement. 
We denote the minimal distance between $R_X$ and $\ltZZi$ by
\begin{equation}
     \Delta_X\coloneqq \inf_{\substack{n\in  \ltZZi,\\ n\neq 0} } \norm{n-P_X n}_{\ell^2(\Z)}
\end{equation}
and the minimal distance between $R_X$ and $\ltZZ$ by
\begin{equation}
     \Delta_{X,\R}\coloneqq \inf_{\substack{n\in  \ltZZ,\\ n\neq 0} } \norm{n-P_X n}_{\ell^2(\Z)}.
\end{equation}
\end{definition}
Our main insight into the continuity problem is the following proposition. 
\begin{proposition}\label{prop:inf_crit}
    Let $X = \{x_k\}_{k\in\Z}\subseteq\R$ be a sampling set for $PW_1(\mathbb{R})$. Let $\mathcal{T}_{1,X}: \foldedRange[1] \to PW_{1}(\R)$ be a left-inverse of $\mathcal{M}_{1,X}$, that is, $\mathcal{T}_{1,X} \mathcal{M}_{1,X} = \id_{PW}$. Then $\mathcal{T}_{1,X}$
    is (uniformly) continuous with respect to the Euclidean metric if and only if $\Delta_X>0$.
    If we restrict $\mathcal{M}_{1,X} $ to
    \begin{equation}
        PW_1(\R,\R) = \{ f\in PW_1(\R): f \text{ real-valued a.e.}\},
    \end{equation}
    then the continuity of the left inverse is equivalent to $\Delta_{X,\R}>0$.
\end{proposition}
\begin{proof}
With the notation in Section \ref{sec:frame_theory}, note that
\begin{equation}\label{eq_yyy2} 
    \dist(c,R_X) = \norm{c - P_X c}_{\ell^2(\Z)},\quad c\in\ell^2(\Z). 
\end{equation}
    By Lemma \ref{lem:continuity_statements} (i), we only need to consider continuity at zero. 

\smallskip

\noindent\textbf{Step 1. Continuity implies a positive distance.}
We prove this by contradiction.
Assume that $\mathcal{T}_{1,X}$ is continuous at $0$, but $\Delta_X=0$. Taking \eqref{eq_yyy2} into account, we can find a sequence $(n_N)_{N\in\N} \subset \ltZZi \setminus \{0\}$ and functions $(f_N)_{N\in\Z}\subset PW_{1}(\R)$ such that 
\begin{equation}\label{eq_yyy1}
    \norm{Cf_N-n_N}_{\ell^2(\Z)}<\tfrac{1}{2N},\quad N\in\N.
\end{equation}
In particular, 
\begin{equation}
    |(C (2 f_N))_k - 2(n_N)_{k} |<\tfrac{1}{N}\leq 1,\quad N\in\N, k\in\Z.
\end{equation}
Hence,
\begin{equation}
    \mathcal{M}_{1, X} (2 f_N) = C (2 f_N)-2 n_N = 2(C f_N- n_N ),\quad N\in\N,
\end{equation}
and 
\begin{equation}\label{eq:M_f_l_bound}
    \norm{\mathcal{M}_{1, X} (2 f_N)}_{\ell^2(\Z)} <\tfrac{1}{N},\quad N\in\N.
\end{equation}
Let us now consider this in terms of the continuity of the inverse mapping. Let $B_X$ denote the upper frame bound of $X$.
    Since the left-inverse map $\mathcal{T}_{1,X}$ is continuous at zero, there exists $\delta>0$ such that for all $h\in PW_1(\R)$
    \begin{equation}
     \norm{\mathcal{M}_{1,X} h}_{\ell^2(\Z)}<\delta
     \quad \text{implies}\quad
       \norm{h}_{L^2(\R)}< B_X^{-1/2}.
    \end{equation}
    We apply this to $2 f_N $ for all $N\geq \delta^{-1}$: Thus, \eqref{eq:M_f_l_bound} implies that 
     \begin{equation}
       \norm{2 f_N}_{L^2(\R)}<B_X^{-1/2},
       \quad N\geq \delta^{-1}.
    \end{equation}
    From the upper frame bound \eqref{eq:sampling_ineq_def}, we conclude that, for $N\geq \delta^{-1}$,    \begin{equation}\label{eq:lem:distequiv_C1}
        \norm{C f_N}_{\ell^2(\Z)} \leq B^{1/2}_X\norm{f_N}_{L^2(\R)} < \tfrac{1}{2},
    \end{equation}
while, by \eqref{eq_yyy1},\begin{equation}\label{eq:lem:distequiv_C2}
    \begin{split}
        \norm{Cf_N}_{\ell^2(\Z)} &\geq \norm{n_N}_{\ell^2(\Z)} - \norm{Cf_N-n_N}_{\ell^2(\Z)} \\
        & \geq 1 - \norm{Cf_N-n_N}_{\ell^2(\Z)}  >1 -\tfrac{1}{2N}\geq \tfrac{1}{2}.
    \end{split}
\end{equation}
This contradiction shows that $\Delta_X>0$.

\medskip

\noindent\textbf{Step 2. Positive distance implies continuity.} Suppose that $\Delta_X>0$, let $A_X$ denote the lower frame bound \eqref{eq:sampling_ineq_def} of $X$ and
    let $\varepsilon>0$. 
    Let $h\in PW_{1}(\R)$ be such that
    \begin{align}
        \norm{\mathcal{M}_{1,X} h}_{\ell^2(\Z)} &< 2\Delta_X , \label{eq:lem:distequiv_B2} \\
        \norm{\mathcal{M}_{1,X} h}_{\ell^2(\Z)}& <A_X^{1/2}\varepsilon, \label{eq:lem:distequiv_B3} 
    \end{align}
    and let us show that $\norm{h}_{L^2(\R)} <\varepsilon$. To this end, let $n\in \ltZZi$ be such that 
    \begin{equation}
        \mathcal{M}_{1,X} h = C h-2 n.
    \end{equation}
    By \eqref{eq:lem:distequiv_B2},
    \begin{equation}
         \norm{C(\tfrac{1}{2} h)- n}_{\ell^2(\Z)} =\tfrac{1}{2}\norm{C h-2 n}_{\ell^2(\Z)} < \Delta_X.
    \end{equation}
    Thus $n=0$ must be the zero sequence and $C h = \mathcal{M}_{1,X} h$. Now the frame inequality \eqref{eq:sampling_ineq_def} and \eqref{eq:lem:distequiv_B3} imply
    \begin{equation}
        \norm{h}_{L^2(\R)} \leq A_X^{- 1/2} \norm{Ch}_{\ell^2(\Z)} 
        = A_X^{- 1/2} \norm{\mathcal{M}_{1,X} h}_{\ell^2(\Z)} 
        <\varepsilon,
    \end{equation}
    as desired.
    
    The same proof gives the statement for real-valued bandlimited functions, as $\re (\mathcal{M}_{\lambda} f )= \mathcal{M}_\lambda\re (f)$. 
    \end{proof}
    \begin{remark}[Critical functions]\label{rem:critical_f}
    If the left inverse 
    $\mathcal{T}_{\lambda,X}$
    is discontinuous with respect to the Euclidean metric, then it is also discontinuous with respect to the toral metric. The proof above tells us how to identify the problematic functions.
    If the sequence of functions $(h_N)_{N\in\N}\subseteq PW_\Omega(\R)$ satisfies 
    \begin{equation}
        \lim\limits_{N\to\infty} \ \inf\limits_{\substack{n\in\ltZZi, \\ n\neq 0}} \norm{n- C_X h_N}_{\ell^2(\Z)} = 0
    \end{equation}
    (the real-valued case is analogous),
    then for any $f\in PW_\Omega(\R)$ the sequence of functions 
    $(f+2\lambda h_N)_{N\in\N}$ satisfies
    \begin{equation}
        \lim\limits_{N\to\infty} \norm{\mathcal{M}_{\lambda,X}f - \mathcal{M}_{\lambda,X} (f+2\lambda h_N)}_{\toral} =0
    \end{equation}
    and $(f+2\lambda h_N)_{N\in\N}$ does not converge to $f$. 
    In particular, 
    when $f=0$ and $\lambda=\frac{1}{2}$, $\mathcal{M}_{\frac{1}{2}, X} h_N$ captures how close the samples $C_X h_N$ are to the integers.
    \end{remark}
    \begin{remark} [Generalizations]\label{rem:general} 
    It is not difficult to see that all results in this section can be stated for sampling in general reproducing kernel Hilbert spaces, or even folded sampling of frame coefficients: 
    If $\Psi = (\psi_k)_{k\in\Z}\subseteq \H$ is a frame, then the folded analysis operator 
    \begin{equation}
        \mathcal{M}_{\lambda,\Psi}: \mathcal{H} \to\ell^2(\Z),\quad f\mapsto 2\lambda\left(   \left\{ \tfrac{\left\langle f,\psi_k\right\rangle_\H}{2\lambda}+\tfrac{1}{2}\right\} -\tfrac{1}{2} \right)_{k\in\Z}
    \end{equation} 
    has a (uniformly) continuous left inverse with respect to the Euclidean metric if and only if    \begin{equation}\label{eq:prop_general}
     \Delta_\Psi\coloneqq \inf_{\substack{n\in  \ltZZi,\\ n\neq 0} } \norm{n-P_\Psi n}_{\ell^2(\Z)}>0.
\end{equation} 
    \end{remark}

\section{Analysis of the distance functional for lattice grids}\label{sec:lattice_tools}

In this section, we restrict our attention to sampling lattices $X=\alpha\mathbb{Z}$ and
exploit the structure of the corresponding Gram matrix $G_{\alpha\Z}$ to better describe the distance functional $\Delta_{\alpha\mathbb{Z}}$.

We use the notation
\begin{equation}
    \F:\ell^2(\Z)\to L^2_{1\text{-per}}(\R),\quad  \F c(t) = \hat{c}(t) =\sum\limits_{k\in\Z}c_k e^{-2\pi i k t},
    \quad t\in\R
\end{equation}
for the discrete Fourier transform, which is a unitary operator from the space of square-summable sequences to the space of $1$-periodic functions which are square-integrable on a period. 
\begin{proposition}\label{prop:PQ_description} 
Consider $X = \alpha \Z\subseteq\R$ with $0<\alpha<1$. Then the projection $P_{\alpha\Z}$ to the range of the analysis operator $C_{\alpha\Z}$ is given by
    \begin{equation}\label{eq_yyy7}
         \left(P_{\alpha\Z}\right)_{j,k}= \alpha\, \sinc(\alpha(j-k)), \qquad j,k\in\Z,
    \end{equation}
    and satisfies 
    \begin{equation}\label{eq:lem:P_description_square}
        \begin{split}
            \norm{P_{\alpha\Z} \, c}_{\ell^2(\Z)}^2 & = \langle \hat{c}, \chi_{(-\alpha/2,\alpha/2)}\hat{c}\,\rangle_{L^2([-1/2,1/2])} = \int_{-\alpha/2}^{\alpha/2} |\hat{c}(t)|^2\, dt,
        \end{split}
    \end{equation}
     where $\sinc (t) =\frac{\sin(\pi t)}{\pi t}$ is the cardinal sine function. 
   Consequently, 
\begin{equation}\label{eq:prop:PQ_description_square}
        \begin{split}
        \Delta_{\alpha\Z} = \inf_{\substack{n\in \ltZZi,\\ n\neq 0} }  \norm{n-P_{\alpha\Z}\, n}_{\ell^2(\Z)} 
        &
        =  \inf_{\substack{n\in \ltZZi,\\ n\neq 0} }  \norm{P_{(1-\alpha)\Z}\, n}_{\ell^2(\Z)},
        \end{split}
    \end{equation}
    and
    \begin{equation}\label{eq:complem_proj}
        \Delta_{\alpha\Z,\R} = \inf_{\substack{n\in \ltZZ,\\ n\neq 0} }  \norm{n-P_{\alpha\Z}\, n}_{\ell^2(\Z)} 
        =  \inf_{\substack{n\in \ltZZ,\\ n\neq 0} }  \norm{P_{(1-\alpha)\Z}\, n}_{\ell^2(\Z)}.
    \end{equation}
\end{proposition}
\begin{proof}
Since $0 < \alpha < 1$, $\alpha \mathbb{Z}$ is a sampling set for $PW_{1}(\R)$. As noted in Section \ref{sec:frame_theory},
    \begin{equation}
        P_{\alpha\Z} = G_{\alpha\Z}G_{\alpha\Z}^\dagger = G_{\alpha\Z}^\dagger G_{\alpha\Z}, \qquad G_{\alpha\Z} =
        \left(\sinc(\alpha(j -k)) \right)_{j,k\in\Z}.
    \end{equation}
    The matrix $G_{\alpha\Z}$ represents the convolution operator with kernel $g_\alpha = \left(\sinc(\alpha k) \right)_{k\in\Z}$. By Poisson's summation formula, for $t\in[-1/2,1/2]$,
\begin{align}
    \widehat{g_\alpha}(t) = \sum\limits_{k\in\Z} \sinc(\alpha k)e^{2\pi i kt} = \alpha^{-1}\sum\limits_{k\in\Z} \chi_{(-1/2,1/2)}(\tfrac{k-t}{\alpha}) = \alpha^{-1} \chi_\alpha(t),
\end{align}
where
$\chi_\alpha$ is the $1$-periodic function that coincides with the indicator function of $(-{\alpha}/{2}, {\alpha}/{2})$
    on $[-1/2,1/2]$, and we used that $0<\alpha<1$. The Gram matrix acts on the Fourier side as the multiplication operator $\mathrm{M}_{\widehat{g_\alpha}} f (t) =\widehat{g_\alpha}(t) f(t)$. Indeed,     \begin{equation}
    \F G_{\alpha\Z} c  = \widehat{G_{\alpha\Z} c} =  \widehat{g_\alpha} \hat{c}= \mathrm{M}_{\widehat{g_\alpha}} \hat{c},\quad c\in\ell^2(\Z).
    \end{equation}
    Thus,
    \begin{equation}
        P_{\alpha\Z} = G_{\alpha\Z}G_{\alpha\Z}^\dagger 
        = (\F^{-1} \mathrm{M}_{\widehat{g_\alpha}} \F) (\F^{-1} \mathrm{M}_{\widehat{g_\alpha}} \F)^{\dagger} 
        =
        \F^{-1} \mathrm{M}_{\widehat{g_\alpha}}  \mathrm{M}_{\widehat{g_\alpha}}^{\dagger} \F
        =\F^{-1} \mathrm{M}_{\chi_\alpha} \F.
    \end{equation}
   This gives \eqref{eq:lem:P_description_square} and 
\begin{equation}     \left(P_{\alpha\Z}\right)_{j,k}= \langle e_j, G_{\alpha\Z}G_{\alpha\Z}^\dagger e_k\rangle_{\ell^2(\Z)} = \int_{-1/2}^{1/2} \chi_\alpha(t)e^{2\pi i (j-k) t} \,dt= \alpha\, \sinc(\alpha(j-k)), 
\end{equation}
as claimed. Finally, by orthogonality and periodicity,
\begin{equation}
    \begin{split}
        \norm{c-P_{\alpha\Z}\, c}_{\ell^2(\Z)}^2 & 
        = \norm{c}_{\ell^2(\Z)}^2 - \norm{P c}_{\ell^2(\Z)}^2 
        = \int_{-1/2}^{-\alpha/2} |\hat{c}(t)|^2\, dt + \int_{\alpha/2}^{1/2} |\hat{c}(t)|^2\, dt
         \\
         &= \int_{1/2}^{1-\alpha/2} |\hat{c}(t)|^2\, dt + \int_{\alpha/2}^{1/2} |\hat{c}(t)|^2\, dt  = 
         \int_{\alpha/2}^{1-\alpha/2} |\hat{c}(t)|^2\, dt
         \\
         & = \int_{-(1-\alpha)/2}^{(1-\alpha)/2} |\hat{c}(t+1/2)|^2\, dt
         = \norm{P_{(1-\alpha)\Z}\, U c}_{\ell^2(\Z)}^2,
    \end{split}
\end{equation}
where $U(c)_k = (-1)^k c_k$, $k\in\Z$.
Since $U$ is a bijective map both on $ \ltZZi$ and $\ltZZ$, taking infimum over $c$ gives \eqref{eq:prop:PQ_description_square} and \eqref{eq:complem_proj}.
\end{proof}

\begin{remark}   
\label{rem:general_proj}
    The first steps of the proof of Proposition \ref{prop:PQ_description} work
    for arbitrary non-equispaced sampling sets $X$ of lower Beurling density $D^-(X) = \alpha^{-1}>1$. However, the corresponding Gram matrix does not have a Toeplitz structure and the projection formula is not as simple as in the case of equispaced sampling. Furthermore, since the frame is not necessarily tight
    (that is the frame bounds may not be equal), the corresponding distance criterion is more complicated.
\end{remark}
\section{Proof of Theorem \ref{thm:main} via prolate matrices and a shortest vector problem}\label{sec:main_via_SVP}
In this section, we interpret the distance criterion as a shortest lattice vector problem and solve it by means of Fourier analysis.
The classical \emph{shortest vector problem (SVP)} (with respect to the $p$-norm) is the problem of determining the shortest non-zero vector of a given lattice $A\Z^N\subseteq \R^N$, $A\in\mathrm{GL}(N,\R)$, $N\in\N$.  That is, we are interested in
\begin{equation}\label{eq:def_svp}
    \mathcal{L}_{A}^p \coloneqq \min_{n\in\Z^N\hspace{-2pt},\, n\neq 0 } \norm{An}_p,\quad \argmin\limits_{n\in\Z^N\hspace{-2pt},\, n\neq 0 } \norm{An}_p.
\end{equation}
The problem was first formulated by Minkowski \cite{Minkowski1968}, who proved:
\begin{equation}\label{eq:Minkowski}
    \mathcal{L}^\infty_{A} \leq |\det A|^{1/N},\qquad \mathcal{L}^p_{A} \leq N^{1/p}|\det A|^{1/N}.
\end{equation}
In our case, the distance criterion (Proposition \ref{prop:inf_crit}) 
suggests an infinite dimensional SVP. The corresponding infinite matrix is the projection $P_{\alpha\Z}$ described in Proposition \ref{prop:PQ_description}. To connect the finite and infinite dimensional theory, we consider certain truncations.

For $\alpha>0$ the \emph{prolate spheroidal wave matrix} $Q_{\alpha, N} \in GL(N,\mathbb{R})$ is
\begin{equation}\label{eq:def_prolate_matrix}
Q_{\alpha, N}= (\alpha \, \sinc(\alpha (j-k)))_{0\leq j,k\leq N-1}
\end{equation}
and was introduced by Slepian in the last installment \cite{Slepian1978} of the Bell Labs series \cite{LandauPollak1961,LandauPollak1962, SlepianPollak1961}. To understand the action of $Q_{\alpha, N}$, consider the operations of extension by $0$ and restriction to the first coordinates:
\begin{align}
&\iota_N:\mathbb{R}^N \to \ell^2(\mathbb{Z}),
\quad
\iota_N(c_0,\ldots,c_{N-1}) = (\ldots,0,c_0, \ldots, c_{N-1}, 0, \ldots),
\\
&\iota^*_N: \ell^2(\mathbb{Z}) \to \mathbb{R}^N,
\quad
\iota^*_N(c) = (c_0,\ldots,c_{N-1}).
\end{align}
With this notation:
\begin{align}\label{eq_yyy5}
Q_{\alpha, N} c = \iota^*_N\big[\iota_N(c) * (\alpha\,\sinc(\alpha k))_{k\in\mathbb{Z}}\big], \qquad c=(c_0,\ldots,c_{N-1}) \in \mathbb{R}^N.
\end{align}
As the convolution in \eqref{eq_yyy5}
is a bandwidth filter, 
\begin{align}\label{eq_yyy8}
\langle Q_{\alpha, N} c, c \rangle
= \int_{-\alpha/2}^{\alpha/2} |\widehat{\iota_N(c)}(t)|^2\, dt, \qquad c \in \mathbb{R}^N,
\end{align}
cf. Proposition \ref{prop:PQ_description}. Hence,
$Q_{\alpha, N}$ enforces a \emph{time-frequency restriction} on a vector $c$: first by truncating its discrete Fourier transform to the bandwidth $[-\alpha/2,\alpha/2]$ and then by restricting the (temporal/spatial) index range to $\{0,\ldots,N-1\}$.

The matrix $Q_{\alpha, N}$ has real entries, is symmetric and positive-semidefinite because of \eqref{eq_yyy8}. 
The spectral profile of $Q_{\alpha, N}$ is well-studied: ordering the eigenvalues of $Q_{\alpha, N}$ as 
\begin{equation}\label{eq_yyy10}
1>\mu_0(Q_{\alpha, N})\geq \mu_1(Q_{\alpha, N})\geq\dots\geq \mu_{N-1}(Q_{\alpha, N})> 0,
    \end{equation}
it turns out that most eigenvalues cluster near either $0$ or $1$, and the transition occurs near the index $k=\alpha N$. We will need the following result by Slepian \cite[eq.~(43--47),~(63)]{Slepian1978}:
given $\alpha, \varepsilon>0$, there exists a constant $\kappa=\kappa(\alpha,\varepsilon)>0$ such that, for sufficiently large $N \in \mathbb{N}$,
\begin{equation}\label{eq:prolate_eigenvalues}       
        \mu_{k}(Q_{\alpha,N}) \leq \mu_{\lfloor \alpha N (1+\varepsilon)\rfloor}(Q_{\alpha,N}) \leq e^{-\kappa N}<1, \quad k  \geq \lfloor \alpha N (1+\varepsilon)\rfloor;
    \end{equation}
see also \cite{MR169015} and \cite{KaRoDa,israel15,BoJaKa,MR4711850,kul2,IsMa,eigenfourier,HuIsMa,HuIsMa2} for recent, more quantitative, results.
    
We are now prepared to prove our first main result.

\begin{proof}[Proof of Theorem \ref{thm:main}]\label{proof:SVP}
    \textbf{Step 1. Reductions.}
    By Lemma \ref{lem:continuity_statements} (iii), it suffices to show that no left-inverse $\mathcal{T}_{\lambda,X}$ exists which is continuous with respect to the Euclidean metric. After this reduction, we further argue that it suffices to consider the case $\Omega=\lambda=1$. Indeed, if $\mathcal{D}_{\Omega}$ and $\mathcal{S}_{\lambda}$ denote the operators
    \begin{align*}  
        \mathcal{D}_{\Omega}: & PW_\Omega(\R)\to PW_1(\R), \quad f\mapsto f(\tfrac{\cdot}{\Omega}), \\
        \mathcal{S}_{\lambda}: & PW_1(\R)\to PW_1(\R), \quad f\mapsto \tfrac{1}{\lambda} f,
    \end{align*}
    we have
     \begin{equation}
        \mathcal{M}_{\lambda, X} f = S_{1/\lambda}\mathcal{M}_{1,\Omega X} \mathcal{S}_\lambda \mathcal{D}_\Omega f,\qquad f\in PW_{\Omega}(\R).
     \end{equation}    
     Since both $\mathcal{D}_{\Omega}$ and $\mathcal{S}_{\lambda}$  are multiples of isometric isomorphisms, we can assume without loss of generality that $\Omega=\lambda=1$.

   Finally, we restrict the range of $\alpha$ as follows. If $\alpha>1$, then $\alpha\mathbb{Z}$ is not a sampling set for $PW_1(\mathbb{R})$ and there exists a non-zero function $f\in PW_1(\mathbb{R})$ such that
   $\mathcal{M}_{1,\alpha\Z} f = C_{\alpha \Z} f=0$. Thus, $\mathcal{M}_{1,\alpha\Z}$ does not admit a left-inverse on its range. If $\alpha=1$, the situation is more subtle. While $\mathbb{Z}$ is a sampling set for $PW_1(\mathbb{R})$, on the set $\mathbb{Z}$, the function $2\,\sinc \in PW_1(\mathbb{R})$ takes values in $2 \mathbb{Z}$ and, thus,
   $\mathcal{M}_{1,\Z} (2\,\sinc) = 0$. Hence, again, $\mathcal{M}_{1,\alpha\Z}$ does not admit a left-inverse on its range.
   From now on, we assume that $0<\alpha<1$ and investigate the continuity. By Theorem \ref{thm_injectivity}, $\mathcal{M}_{1,\alpha\Z}$ does admit a left-inverse 
   $\mathcal{T}_{1,\alpha\Z}: \mathrm{Range}(\mathcal{M}_{1,\alpha\Z}) \to PW_1(\mathbb{R})$, but we shall disprove its continuity with respect to the Euclidean metric.
   
  \medskip
    
    \noindent\textbf{Step 2. The shortest vector problem.}
 By Proposition \ref{prop:PQ_description},
    \begin{align}
        \Delta_{\alpha \Z} =\inf_{\substack{n\in  \ltZZi,\\ n\neq 0} } \norm{n-P_X n}_{\ell^2(\Z)} \leq \Delta_{\alpha \Z,\R} &=\inf_{\substack{n\in  \ltZZ,\\ n\neq 0} } \norm{n-P_X n}_{\ell^2(\Z)} 
        \\
        &=\inf\limits_{\substack{ n\in \ltZZ \\n \neq 0}} \norm{P_{(1-\alpha)\Z}n }_{\ell^2(\Z)}.
    \end{align}
    By Proposition \ref{prop:inf_crit},
    in order to disprove the continuity of $\mathcal{T}_{1,\alpha\Z}$ with respect to the Euclidean metric, it suffices to show that $\Delta_{\alpha\Z,\R} =0$.
    
    Let $c\in \ell^2(\Z)$ be a finitely supported sequence, with $c_k =0$ whenever $k< k_0$ or $k> k_0+N-1$. 
Let $d=\iota^*_N (c_{k_0+k})_{k\in\Z}$, i.e.,  $d_k=c_{k_0+k}$, $0\leq k\leq N-1$.
Then the projection \eqref{eq_yyy7} and the prolate matrix \eqref{eq:def_prolate_matrix} are related by
\begin{equation}
    \norm{P_{(1-\alpha)\Z} c}_{\ell^2(\Z)}^2 
    = \langle c, P_{(1-\alpha) \Z} c\rangle _{\ell^2(\Z)} 
    = d\cdot Q_{1-\alpha, N} d = \norm{Q_{1-\alpha, N}^{1/2} d}_{2}^2,
\end{equation}
where $Q_{1-\alpha, N}^{1/2} \in \mathbb{R}^{N\times N}$ is the symmetric positive square root of $Q_{1-\alpha, N}^{1/2}$. Hence,
    \begin{equation}
        \Delta_{\alpha\Z,\R} =\inf\limits_{\substack{ n\in \ltZZ \\n \neq 0}} \norm{P_{(1-\alpha)\Z}n }_{\ell^2(\Z)}
        =\lim\limits_{N\to\infty} \inf\limits_{ n\in \Z^N} \norm{Q_{1-\alpha, N}^{1/2}n }_{2}.
    \end{equation}
     The inner infimum is monotone in $N$, and equals the shortest lattice distance $\mathcal{L}_{Q_{1-\alpha, N}^{1/2}}^2$. Hence, by Minkowski's bound \eqref{eq:Minkowski}:
    \begin{align*}
        \Delta_{\alpha\Z,\R} = \lim\limits_{N\to\infty} \mathcal{L}_{Q_{1-\alpha, N}^{1/2}}^2
        \leq \liminf\limits_{N\to\infty}
        \sqrt{N} \big(\det Q_{1-\alpha, N}^{1/2}\big)^{\tfrac{1}{N}}
        = \liminf\limits_{N\to\infty}
        \sqrt{N} (\det Q_{1-\alpha, N})^{\tfrac{1}{2N}}.
    \end{align*}

\medskip
    
\noindent\textbf{Step 3. Prolate eigenvalue estimates.}
We consider the eigenvalues \eqref{eq_yyy10}
of $Q_{1-\alpha, N}$. Let $\varepsilon>0$ be small enough such that $(1-\alpha)(1+\varepsilon)<1$, and $N$ sufficiently large so as to apply \eqref{eq:prolate_eigenvalues}. We set $k_0 =\lfloor (1-\alpha)N (1+\varepsilon)\rfloor$ and consider the constant $\kappa = \kappa(1-\alpha,\varepsilon)$ from \eqref{eq:prolate_eigenvalues}.
    Then
    \begin{align*}
        &\sqrt{N} (\det Q_{1-\alpha, N})^{\frac{1}{2N}} \leq \sqrt{N}\Big( \prod\limits_{l= k_0}^{N-1} e^{-\kappa N}\Big)^{\frac{1}{2N}}
        \leq \sqrt{N} \left( e^{-\kappa N}\right)^{\frac{N-k_0}{2N}} \\
        & \qquad \leq \sqrt{N} \left( e^{-\kappa N}\right)^{\frac{N-(1-\alpha)N(1+\varepsilon)}{2N}}
         =        
        \sqrt{N} \big(e^{-\kappa\tfrac{1-(1-\alpha)(1+\varepsilon)}{2}}\big)^N \longrightarrow 0,
    \end{align*}
    as $N \to \infty$. This shows that $\Delta_{\alpha\Z,\R}=0$.
\end{proof}

\section{A second proof of Theorem \ref{thm:main}}
\label{sec:main_via_TP}
In this section, we provide an alternative proof of Theorem \ref{thm:main} which is still based on the distance criterion and the analysis of the distance functional for sampling grids, but
avoids the eigenvalue estimates of the prolate spheroidal wave matrix and Minkowski's bound, at least directly. 

\subsection{An explicit example}
We begin by demonstrating the approach with an explicit example.  

\begin{theorem}\label{thm:two_thirds_binom}
    Let ${\frac{2}{3}}<\alpha <1$. Then $\Delta_{\alpha \Z,\R} = \Delta_{\alpha \Z} =0$.
\end{theorem}
\begin{proof}
    For $N \in \mathbb{N}$ let\begin{equation}\label{eq:def:binom_coeff}
        n_k=n_{N,k} =\begin{cases}
        (-1)^{N-k}\binom{N}{N-k},&\quad 0\leq k\leq N,\\
        0,&\quad \text{otherwise.}
        \end{cases}
    \end{equation}
    Its Fourier transform is given by
    \begin{align*}
        \hat{n}(t) 
        & = \sum\limits_{k=0}^N \begin{psmallmatrix}N \\ N-k \end{psmallmatrix}
        e^{2\pi i kt}(-1)^{N-k} = \sum\limits_{k=0}^N \begin{psmallmatrix}N \\ k \end{psmallmatrix}e^{2\pi i (N-k)t}(-1)^{k}  
        = (e^{2\pi i  t}-1)^N \\
        & = \left(2i \, e^{\pi i t}\cdot\tfrac{e^{\pi i t} -e^{-\pi i t}}{2i}\right)^N
        = (2 i\, e^{\pi i t})^N \sin(\pi t)^N.
    \end{align*}
    We aim to apply Proposition \ref{prop:PQ_description}, specifically, \eqref{eq:complem_proj} in terms of $P_{(1-\alpha)\Z}$ to $n_N$:
      \begin{align*}
        \norm{P_{(1-\alpha)\Z}\,n}_{\ell^2(\Z)}^2
        = 2^{2N}\int_{-(1-\alpha)/2} ^{(1-\alpha)/2} \sin({\pi t})^{2N}\, dt.
    \end{align*}
    As the integrand is even and strictly monotonically increasing on $(0, \frac{1-\alpha}{2})\subseteq (0,\frac{\pi}{2})$, 
     \begin{align*}
    \norm{P_{(1-\alpha)\Z}\,n}_{\ell^2(\Z)}^2
    & = 2^{2N+1}\int_{0} ^{(1-\alpha)/2} \sin({\pi t})^{2N}\, dt \leq 2^{2N+1}\int_{0} ^{(1-\alpha)/2} \sin\left(\tfrac{\pi(1-\alpha)}{2}\right)^{2N}\, dt \\
    & = (1-\alpha) \left(2\sin\big(\tfrac{\pi(1-\alpha)}{2}\big)\right)^{2N}.
    \end{align*}
     Since $\frac{2}{3}<\alpha$, we have $\frac{\pi(1-\alpha)}{2}< \frac{\pi}{6}$, and, therefore, 
    $\sin(\frac{\pi(1-\alpha)}{2}) < \sin(\frac{\pi}{6}) = \frac{1}{2}$.
    Hence, $\norm{P_{(1-\alpha)\Z}\,n}_{\ell^2(\Z)} \to 0$, as $N \to \infty$, as desired.
\end{proof}
With the notation of the previous proof, the bandlimited functions with samples $P_{\alpha\Z} n_N$ can be computed as $ \alpha \,C_{\alpha\Z}^*\,n_N$.
By Remark \ref{rem:critical_f}, they are functions with samples arbitrarily close to the integers and indicate the discontinuity of $\mathcal{M}_{\frac{1}{2},\alpha\Z}$. We plot a few of these functions in Figures \ref{fig_xx} and \ref{fig:norm_compare}. For better visualization, we choose $N$ even
and re-center them as
$f_N = \alpha \, C_{\alpha\Z}^* \, m_{N}$ with $m_{N,k}= n_{N, k+N/2}$.
\begin{figure}[h]
\vspace{-10pt}
\subfigure[$f_N =\alpha C_{\alpha \Z}^* m_4$, $\alpha=0.7$.]{
\includegraphics[width=.3\textwidth]{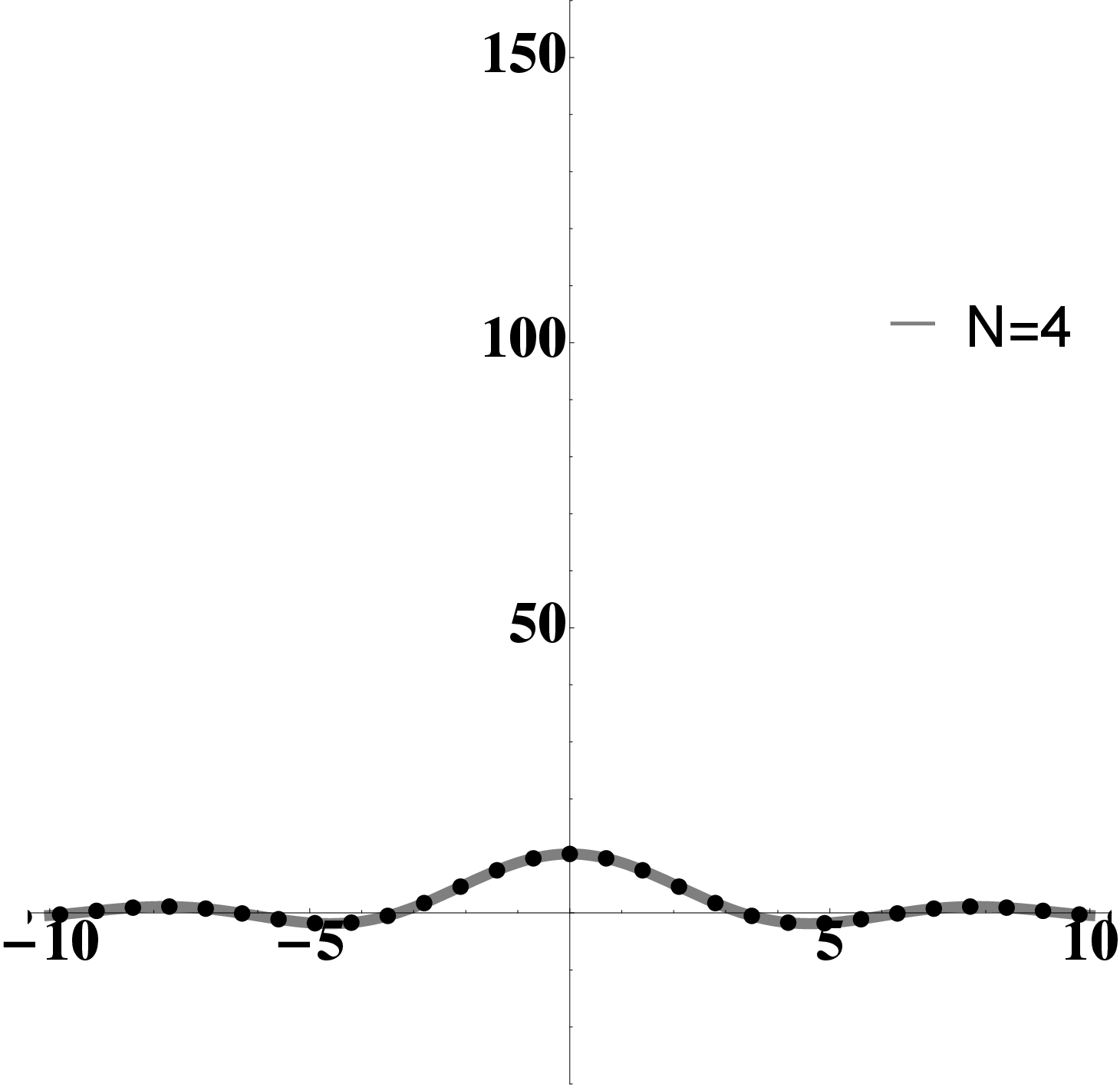}
	}
	\hfill
	\subfigure[$f_N =\alpha C_{\alpha \Z}^* m_6$, $\alpha=0.7$.]{
\includegraphics[width=.3\textwidth]{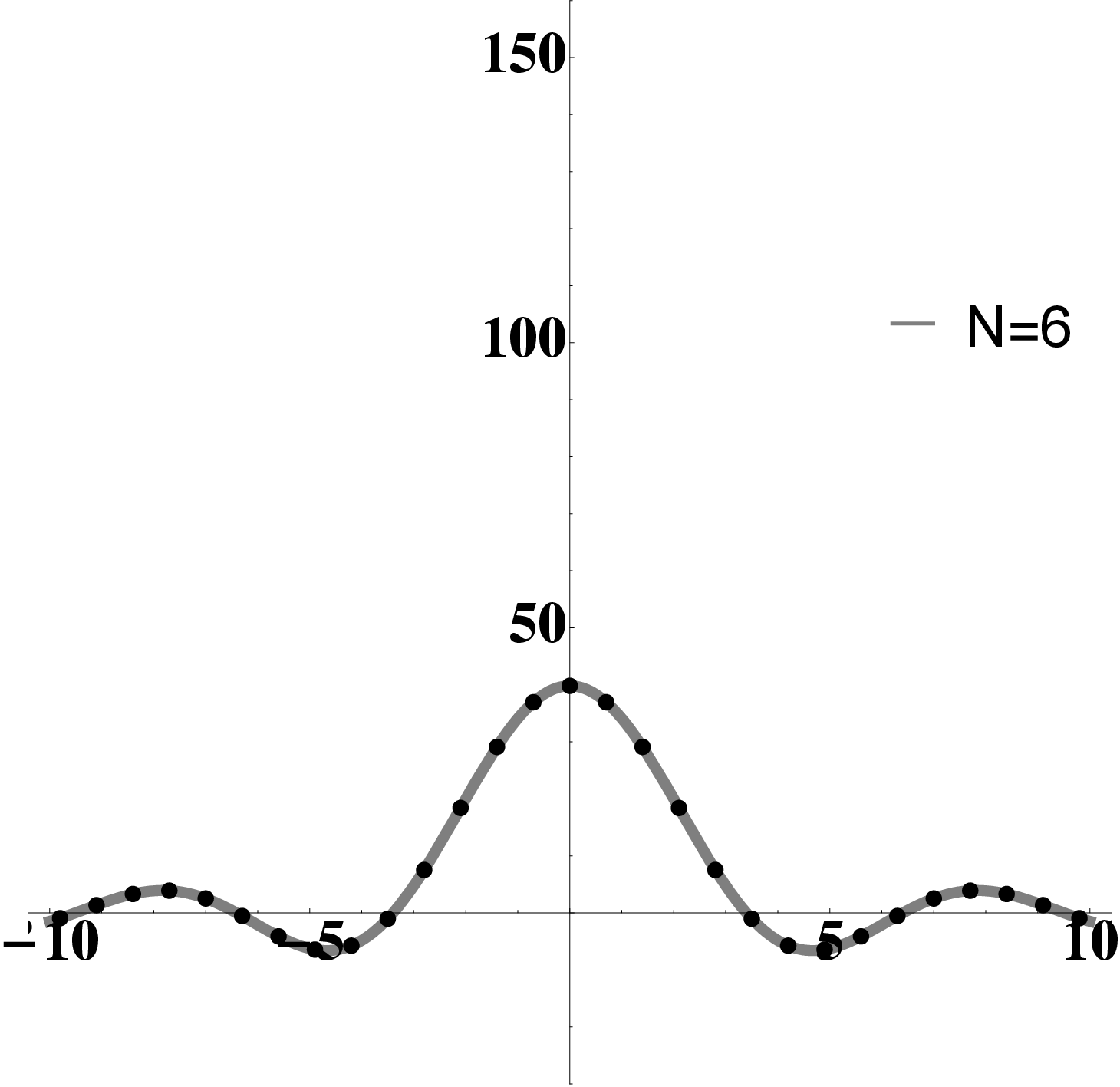}
	}
    \hfill
	\subfigure[$f_N =\alpha C_{\alpha \Z}^* m_8$, $\alpha=0.7$.]{
\includegraphics[width=.3\textwidth]{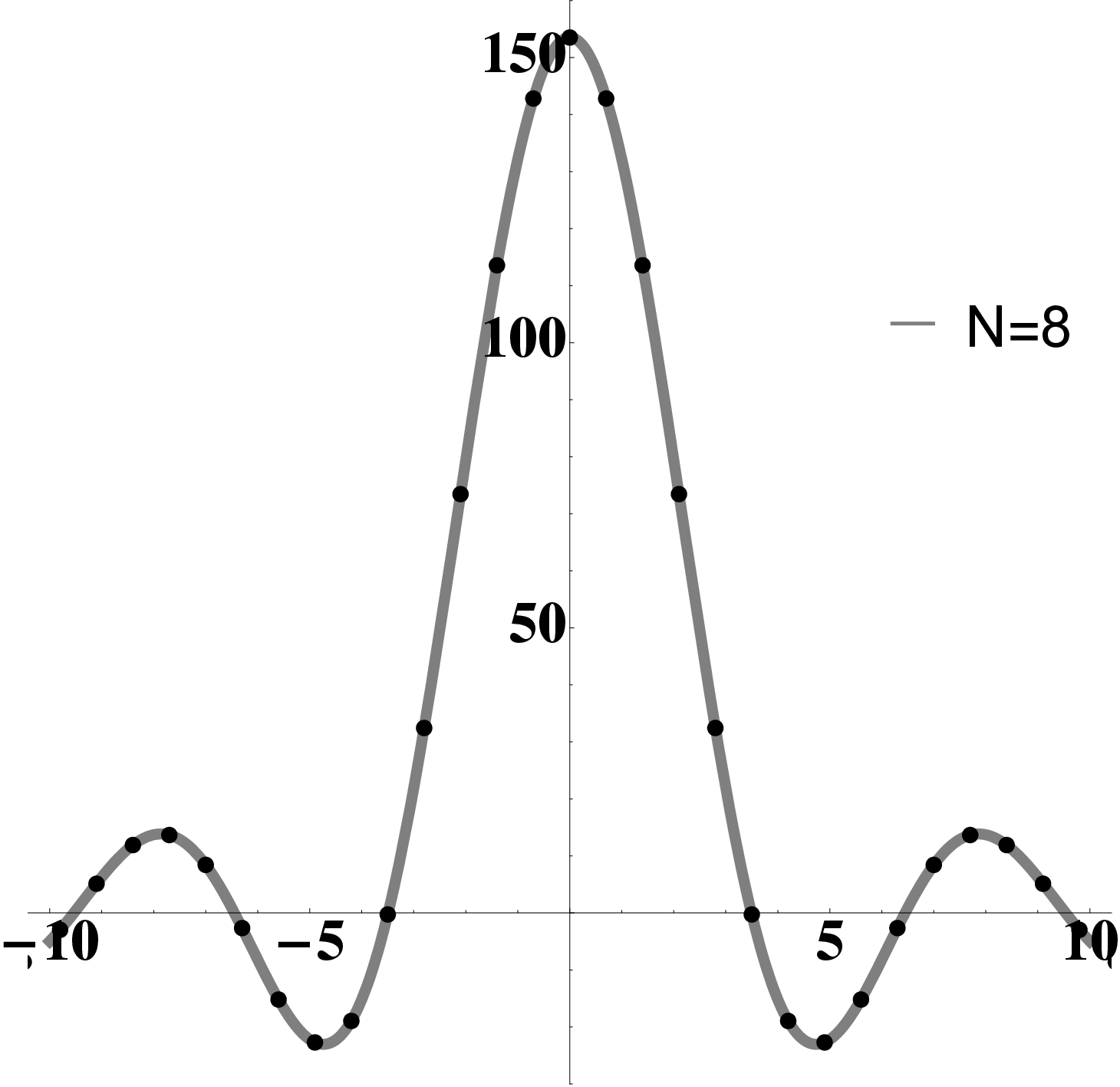}
	}
    \vspace{-10pt}
\caption{Bandlimited functions whose samples are given by the sequences in the proof of Theorem \ref{thm:two_thirds_binom}.
}
\label{fig_xx}
\end{figure}

\begin{figure}[h] 
	\subfigure[$f_N =\alpha C_{\alpha \Z}^* m_N$, $\alpha=0.7$.
    ]{
\includegraphics[width=.35\textwidth]{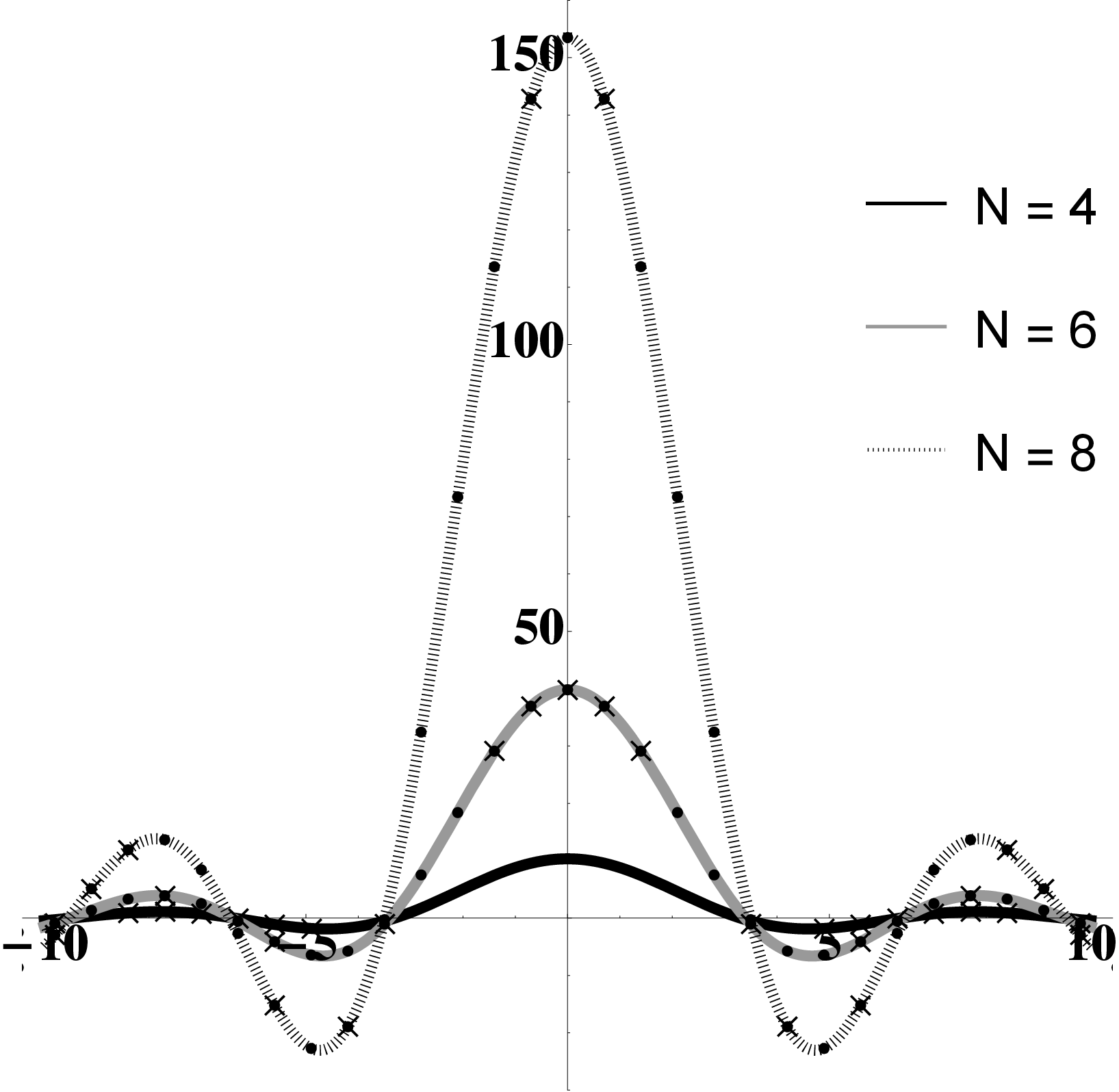}
	}\hspace{1em}
	\subfigure[Log-plot of norms]{
\includegraphics[width=.38\textwidth]{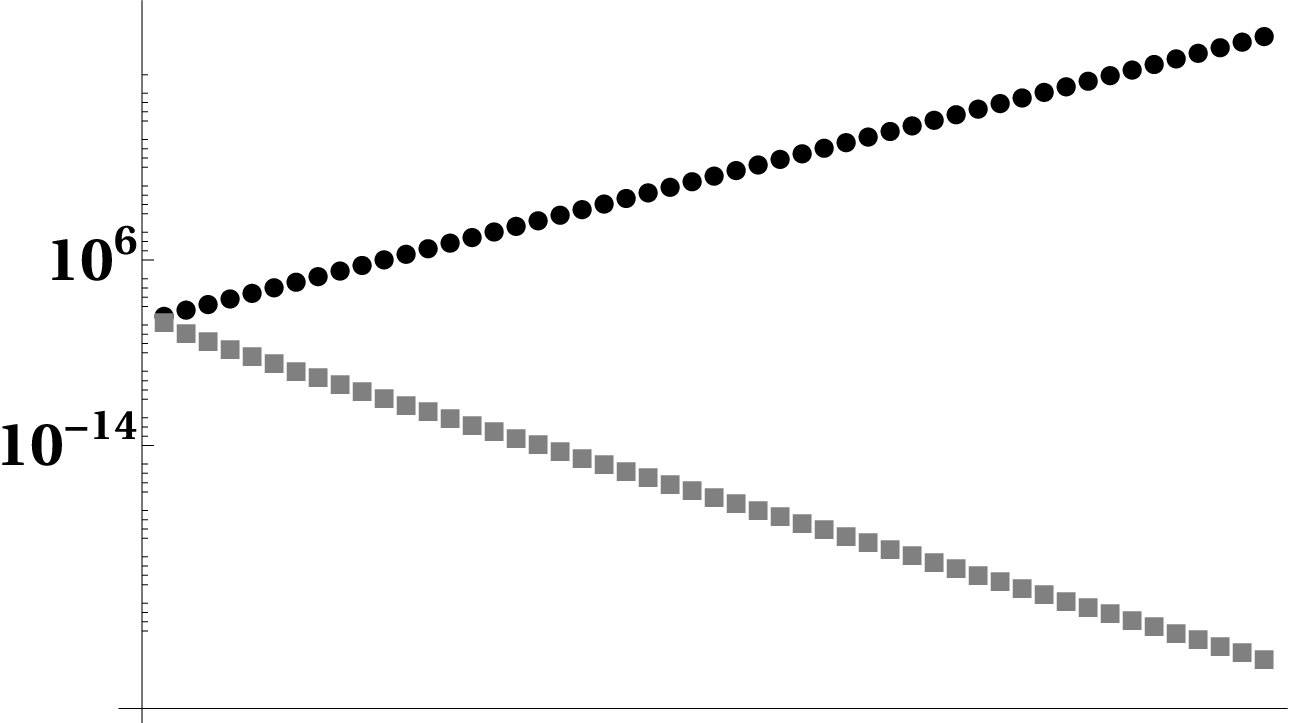}
	}
    \caption{Left: the functions $f_N$, sampled at $0.7\Z$ (black dots) and the measurements with $|\mathcal{M}_{\lambda}f(0.7 k)|<0.25$, $\lambda=\frac{1}{2}$ (crosses). Right: log-plot of the norms of $P_{\alpha\Z} m_N$ (black circles) and the norms of $n-P_{\alpha\Z} m_N$ (gray squares) $1\leq N\leq 30$. }\label{fig:norm_compare}
\end{figure}

\subsection{Proof of Theorem \ref{thm:main} via Tschebyschev polynomials}

We elaborate on the previous example by looking for trigonometric polynomials with integer coefficients that are well-concentrated on a given interval.

Recall that the monic Tschebyschev  polynomial $T_m$ minimizes the supremum norm on a given interval $[a,b]$ among all monic polynomials of degree $m$ \cite[Thm.~2.1.]{Rivlin2020}. The \emph{integer Tschebyschev polynomials} 
$T^{\mathbb{Z}}_m$ on $[a,b]$ are non-zero polynomials with integer coefficients which minimize the supremum norm among all non-zero polynomials with integer coefficients of fixed degree. The integer restriction 
makes the minimization problem much more subtle, because the problem is not invariant under arbitrary translations and dilations. Since
\begin{equation}
    \norm{T^{\mathbb{Z}}_m}_{L^\infty(a,b)}\geq  \norm{T_m}_{L^\infty(a,b)} = 2 \tfrac{(b-a)^m}{4^m},
\end{equation}
one can only hope to obtain a small 
supremum norm if $b-a<4$. Under this condition, Kashin's bound \cite{Kashin1991} 
- which improves on a similar estimate by Fekete \cite{Fekete1923} - states that there exists a universal constant $C>0$ such that 
\begin{equation}\label{eq:kashin}
    \norm{T^{\mathbb{Z}}_m}_{L^\infty(a,b)} \leq C\, m^{1/2}\left( \tfrac{b-a}{4} \right)^{m/2}.
\end{equation}
This allows us to give a second proof of Theorem \ref{thm:main}.
\begin{proof}[Second proof of Theorem \ref{thm:main}]\label{proof:Tschbyschev}
 We proceed as in Step 1 of the proof in Section \ref{sec:main_via_SVP}: We consider $\Omega=\lambda=1$ and $0<\alpha<1$, 
    invoke Propositions \ref{prop:PQ_description} and  \ref{prop:inf_crit} and analyze
    \begin{align}
        \Delta_{\alpha \Z,\R} &=\inf_{\substack{n\in  \ltZZ,\\ n\neq 0} } \norm{n-P_X n}_{\ell^2(\Z)} =\inf\limits_{\substack{ n\in \ltZZ \\n \neq 0}} \norm{P_{(1-\alpha)\Z}n }_{\ell^2(\Z)}
        \\
        &=        \Big[\inf\limits_{\substack{ n\in \ltZZ \\n \neq 0}}
        \int_{-(1-\alpha)/2}^{(1-\alpha)/2} |\hat{n}(t)|^2\, dt\Big]^{1/2}. 
    \end{align}
    We want to show that $\Delta_{\alpha \Z,\R}=0$. Let $N \in \mathbb{N}$, $T_N^{\mathbb{Z}}$ an integer Tschebyschev polynomial, and 
\begin{align}
P_N(t)=T_N^{\mathbb{Z}}(2\cos(2\pi t))= T_N^{\mathbb{Z}}(e^{2\pi i t}+e^{-2\pi i t}) = \sum\limits_{k=-N}^N n_{N,k} e^{2\pi i kt}
\end{align}
the corresponding trigonometric polynomial, whose coefficients $n_{N,k}$ are also integers. 
We estimate
\begin{align}
&\int_{-(1-\alpha)/2}^{(1-\alpha)/2} |\widehat{n_{N}} (t)|^2 dt
=
\int_{-(1-\alpha)/2}^{(1-\alpha)/2} |T_N^{\mathbb{Z}}(2\cos(2\pi t))|^2 dt
\\
&\qquad\leq
(1-\alpha)\sup \left\{|T_N^{\mathbb{Z}}(x)|^2 : x = 2\cos(\pi t), \, |t|\leq 1-\alpha\right\}.
\end{align}
Since $0<1-\alpha<1$, 
\begin{equation}
    \{ x = 2\cos(\pi t): \, |t|\leq 1-\alpha\} = [2\cos(\pi (1-\alpha)), 2] \subseteq (-2,2],
\end{equation}
and we can apply the Fekete-Kashin bound \eqref{eq:kashin} with $a = 2\cos(\pi (1-\alpha))$, $b = 2$. Since $0<b-a<4$, this implies that

\begin{align}
\int_{-(1-\alpha)/2}^{(1-\alpha)/2} |\widehat{n_{N}} (t)|^2 dt \leq C^2 (1-\alpha)\, N\left( \tfrac{b-a}{4} \right)^{N} \longrightarrow 0,
\end{align}
as $N \to \infty$, and $\Delta_{\alpha \Z,\R}=0$.
\end{proof}
The two proofs of Theorem of \ref{thm:main} that we presented 
rely on different, but closely related techniques; see \cite{Kashin1991}.

\section{Continuous reconstruction on bounded sets}\label{sec:bounded_case}

We investigate the continuity of the left-inverse of the folded sampling operator when the sampling problem is supplemented with an a priori bound
on the size of the inputs. For equispaced sampling configurations, 
the algorithms developed in \cite{BhandariKrahmerPoskitt2022} imply that a priori bounds on the size of inputs lead to a certain stability guarantee for the folded reconstruction problem. We shall stablish this fact also for non-uniform sampling configurations, and, as corollary, extend the injectivity result \eqref{thm_injectivity} to arbitrary sampling geometries.

\begin{proof}[Proof of Theorem \ref{thm:main_bounded}]
We begin by first considering the case where one of the functions is the zero function. We establish the inequality and then apply it to the difference $h=f-g$.

\noindent\textbf{Step 1. Removing peaks.}
    Let $h\in PW_{\Omega}(\R)$ with $\norm{h}_{L^2(\R)}<M$ be given. 
    We split the samples as $C_X h=a+\mathcal{M}_{\lambda,X} h$.     
    Since $X$ is a sampling set, by the triangle inequality and the upper frame bound
    \begin{equation}
        \norm{a}_{\ell^2(\Z)} \leq \norm{C_X h}_{\ell^2(\Z)} +\norm{\mathcal{M}_{\lambda,X} h}_{\ell^2(\Z)}\leq B_X^{1/2}\norm{h}_{L^2(\R)} +\norm{\mathcal{M}_{\lambda,X} h}_{\ell^2(\Z)}.
    \end{equation}
Note that
\begin{equation}
    \norm{\mathcal{M}_{\lambda,X} h}_{\ell^2(\Z)} \leq \norm{C_X h}_{\ell^2(\Z)} \leq B^{1/2}_X\norm{h}_{L^2(\R)},
\end{equation} 
so $\norm{a}_{\ell^2(\Z)}$ can be further bounded independently of $\norm{h}_{L^2(\R)}$ by
\begin{equation}
    \norm{a}_{\ell^2(\Z)} \leq 2B_X^{1/2}M.
\end{equation}
The sequence $a$ takes on values in $2\lambda\, \Z[i]$, hence we can bound the number of elements of $F_h = \{x_k\in X: a_k\neq 0 \}$ by
\begin{equation}\label{eq_yyy12}
    \# F_h \leq \sum\limits_{k\in\Z} \left|\tfrac{a_k}{2\lambda} \right|^2 = (2\lambda)^{-2}\norm{a}_{\ell^2(\Z)}^2
    \leq \lambda^{-2} B_X M^2.
\end{equation}
By construction of the set $F_h$, $h(x_k) = \mathcal{M}_\lambda h(x_k)$ for all $x_k\in X\setminus F_h$. 
Note that $D^-(X\setminus F_h) =D^-(X)>\Omega$, so we can apply the frame inequality for $X\setminus F_h$ and Lemma \ref{lem:toral_translation_inv} to obtain 
\begin{equation}
    A_{X\setminus F_h}\norm{h}^2_{L^2(\R)} \leq \norm{\mathcal{M}_{\lambda,X}h}_{\ell^2(\Z)}^2 =
    \norm{\mathcal{M}_{\lambda,X} h}_{\toral}^2 .
\end{equation}
To prove the claim, though, we need to prove that a positive lower frame bound $A_{M,\lambda, X}\leq A_{X\setminus F_h}$ can be established independently of $h$.

\medskip

\noindent\textbf{Step 2. A uniform lower bound.}
We aim to find $A_{M,\lambda,X}>0$ with
\begin{equation}\label{eq:h_ineq}
       A_{M,\lambda, X} \norm{h}^2_{L^2(\R)} \leq \norm{\mathcal{M}_{\lambda,X} h}_{\ell^2(\Z)}^2
\end{equation}
by showing that the quantities in Theorem \ref{thm:bound_dependency} can be established for $X\setminus F_h$ independently of $h$. 
To this end, let
\begin{equation}
    0<\delta = \inf\limits_{\substack{x,x'\in X, \\ x\neq x'}}|x-x'| \leq \inf\limits_{\substack{x,x'\in X\setminus F_h, \\ x\neq x'}}|x-x'|.
\end{equation}
Thus, the separation parameter can be chosen independently of $h$, and it remains to establish \eqref{eq:choice_r}.
We select $0< \varepsilon<\tfrac{1}{2}(D^-(X)-\Omega)$. 
Due to the definition of the lower Beurling density \eqref{eq:def:lower_beurling}, we can choose $r_X>0$  such that 
\begin{equation}
    \inf\limits_{t\in\R} \tfrac{\#X\cap [t-r,t+r]}{2r} >\Omega+2\varepsilon, \quad r\geq r_X.
\end{equation}
Using \eqref{eq_yyy12}, for $ t\in\ \R$, and $r\geq r_X$,
\begin{equation}
    \tfrac{\#(X\setminus F_h)\cap [t-r,t+r]}{2r} \geq \tfrac{(\#X\cap [t-r,t+r]) - \#F_h}{2r}
    \geq \Omega+2\varepsilon -\tfrac{B_X M^2}{2\lambda^{2} r}.
\end{equation}
To bound the last expression from below by $\Omega+\varepsilon$, it suffices to consider $r$ large enough, more precisely, 
\begin{equation}\label{eq:quant_r}
    r\geq \max\left\lbrace r_X,  \tfrac{B_X M^2}{2\lambda^2\varepsilon}\right\rbrace.
\end{equation} 
Hence, by Theorem 
\ref{thm:bound_dependency}
\begin{equation}
    0< A\left( 
    \Omega, \delta, \varepsilon, \max\left\lbrace r_X,  \tfrac{B_X M^2}{2\lambda^2\varepsilon}\right\rbrace
    \right)\eqqcolon A_{M,\lambda, X}\leq A_{X\setminus F_h}.
\end{equation}

We first establish the injectivity of $\mathcal{M}_{\lambda,X}$ on $PW_\Omega(\R)$ in order to be able to discuss a left inverse.

\medskip

\noindent{\hypertarget{step:bounded_inj}{\textbf{Step 3. Global injectivity.}}}
  Let $f,g\in PW_\Omega(\R)$ such that $\mathcal{M}_{\lambda,X} f = \mathcal{M}_{\lambda,X} g$, and set $h = f-g\in PW_\Omega(\R)$. 
  By Lemma \ref{lem:toral_translation_inv}, 
  \begin{equation}
      \mathcal{M}_{\lambda,X} h = \mathcal{M}_{\lambda,X} f -\mathcal{M}_{\lambda,X} g = 0.
  \end{equation}
  Let $M =\norm{h}_{L^2(\R)}$ and $\mathcal{B}$ be the closed ball of radius $M$ in $PW_\Omega(\R)$. 
  By Step 2, $h=0$, and $f=g$.

\medskip

\noindent\textbf{Step 4. Lipschitz continuity.} Let $f,g\in \mathcal{B}$. Then $h=f-g\in 2\mathcal{B}$, and by Lemma \ref{lem:toral_translation_inv},
\begin{equation}
    \norm{\mathcal{M}_{\lambda,X} f-\mathcal{M}_{\lambda,X} g}_{\toral} = \norm{\mathcal{M}_{\lambda,X} h}_{\toral} = \norm{\mathcal{M}_{\lambda,X} h}_{\ell^2(\Z)}.
\end{equation}
By Step 2, 
\begin{equation}
    \norm{f-g}_{L^2(\R)} =  \norm{h}_{L^2(\R)} \leq A_{2M,\lambda,X}^{-1/2} \norm{\mathcal{M}_{\lambda,X} h}_\toral =A_{2M,\lambda,X}^{-1/2}  \norm{\mathcal{M}_{\lambda,X} f-\mathcal{M}_{\lambda,X} g}_{\toral}.
\end{equation}
This concludes the proof. 
\end{proof}
\begin{remark}
    Since $|z-w|_\lambda \leq |z-w|$ for all $z,w\in\C$ and all $\lambda>0$, in the context of Theorem \ref{thm:main_bounded},
    the left-inverse $\mathcal{T}_{\lambda,X}$ of $\mathcal{M}_{\lambda,X}$ is Lipschitz continuous with respect to the Euclidean metric: 
    \begin{equation}
        \norm{f-g}_{L^2(\R)} \leq C \norm{\mathcal{M}_{\lambda,X} f - \mathcal{M}_{\lambda,X} g}_{\ell^2(\Z)}.
    \end{equation}
\end{remark}
  Theorem \ref{thm:main_bounded} reestablishes the injectivity results in \cite{BhandariKrahmer2019,RomanovOrdentlich2019} and extends their scope to non-equispaced sampling geometries. The proof also shows that, as the folding threshold $\lambda$ grows, the importance of the a priori bound $M$ on the input size decreases, cf. \eqref{eq:quant_r}, and, in the limit
  $\lambda \to \infty$, it disappears.

\section{Conclusions and open problems}\label{sec_con} 
We showed that, without a priori energy bounds, the problem of recovering bandlimied functions from folded (modulo) samples taken on equispaced grids is ill-conditioned. The instability result brings context to the recent literature in unlimited sampling, showing that there is room for a more quantitative analysis of the various injectivity or encoding guarantees. We contributed to this topic by showing that, even for non-uniform sampling geometries, where sampling-reconstruction formulas are not explicit, any a priori energy bound regularizes the folded sampling problem. Our proof suggests a general trade-off 
in the stability margins between the folding threshold and the a priori energy bounds. A precise quantification of this observation, and an extension of the instability result to arbitrary geometries remain open problems (cf. Remark \ref{rem:general_proj}).
  
\section{Proof of Theorem \ref{thm:bound_dependency}}\label{proof_samp}
We only sketch the argument and refer the reader to \cite{BeurlingCollected1989} for more details on the techniques. We consider the \emph{Bernstein space} 
\begin{align}
PW_\Omega^\infty(\mathbb{R})
=\big\{f \in L^\infty(\mathbb{R}): \mathrm{supp}(\hat{f}) \subset [-\Omega/2,\Omega/2]\big\},
\end{align}
defined similarly to the Paley-Wiener space \eqref{eq:SWK}, but with respect to the $L^\infty$-norm. In the definition, the Fourier transform should be interpreted distributionally. By the Paley-Wiener theorem, each function $f \in PW_\Omega^\infty(\mathbb{R})$ extends to an entire function (on $\mathbb{C}$) but we shall not need this fact.

\smallskip

\noindent {\bf Step 1: Sampling in the Bernstein space.} We show that there exists a constant $\beta^\infty(\Omega,\delta,\varepsilon,r)$ such that
\begin{align}\label{eq_proof_0}
||f||_\infty \leq \beta^\infty(\Omega,\delta,\varepsilon,r) \sup_{k \in \mathbb{Z}} |f(x_k)|.
\end{align}
We use the notation $C_X f = (f(x_k))_{k\in\mathbb{Z}}$ also for $f \in PW^\infty_\Omega(\mathbb{R})$.
Fix $\delta, \varepsilon, r>0$ and, for the sake of contradiction, assume
that no adequate constant $\beta^\infty(\Omega,\delta,\varepsilon,r)$ can be found. Then there exist a sequence of separated sets $X_N=\{x_{N,k}:k\in\mathbb{Z}\}$ such that     \begin{equation}\label{eq:X_N_initial}
    \inf\limits_{j,k\in\Z, j\neq k}|x_{N,j}-x_{N,k}|\geq \delta >0,
    \qquad
    \inf\limits_{t\in\R} \tfrac{\#X_N\cap [t-r,t+r]}{2r} \geq \Omega+\varepsilon, 
\end{equation}
    and a sequence of functions $\{f_N\}_{N\in\N} \subset PW^\infty_\Omega(\R)$ such that $\norm{f_N}_{L^\infty(\R)}=1$ and
    \begin{equation}        \lim\limits_{N\to\infty}\norm{C_{X_N}f_N}_{\ell^\infty(\Z)} = 0.
    \end{equation}    
    Due to the normalization, for each $N$ there exists a point $z_N\in\R$ such that $|f_N(z_N)|\geq \frac{3}{4}$. Let us set $g_N = f_N(\cdot+z_N)\in PW_\Omega^\infty(\R)$ and $Y_N=X_N-z_N=\{y_{N,k}:k\in\mathbb{Z}\}=\{x_{N,k}-z_N:k\in\mathbb{Z}\}$. Then, $\|g_N\|_\infty=1$,
    \begin{equation}    \lim\limits_{N\to\infty}\norm{C_{Y_N} g_N}_{\ell^\infty(\Z)}
    = \lim\limits_{N\to\infty}\norm{C_{X_N} f_N}_{\ell^\infty(\Z)} =0,
    \end{equation}    \begin{equation}\label{eq:g_N_int_bound}
    |g_N(0)| = |f_N(z_N)|\geq \tfrac{3}{4},
    \end{equation}
    and
    \begin{align}\label{eq_proof_1}
 \inf\limits_{j,k\in\Z, j\neq k}|y_{N,j}-y_{N,k}|\geq \delta >0,
    \qquad
    \inf\limits_{t\in\R} \tfrac{\#Y_N\cap [t-r,t+r]}{2r} \geq \Omega+\varepsilon. 
    \end{align}
    Since $\|g_N\|_\infty =1$, by passing to a subsequence we may assume that $g_N \to g \in PW^\infty_\Omega(\mathbb{R})$ is
    the $\sigma(L^\infty,L^1)$ topology.
    Similarly, we can form the atomic measures
    $\sum_{k\in\mathbb{Z}} \delta_{y_{N,k}}$ and use \eqref{eq_proof_1} to pass to a subsequence and extract a \emph{weak limit set} $Y=\{y_k:k\in\mathbb{Z}\} \subset \mathbb{R}$ such that
$\sum_{k\in\mathbb{Z}} \delta_{y_{N,k}} \to 
\sum_{k\in\mathbb{Z}} \delta_{y_k}$ in the $\sigma(C_c^*,C_c)$ topology, see, e.g. \cite[p.~344--347]{BeurlingCollected1989} or \cite[Sec.~4]{MR3336091}. When passing to the limit, the uniform properties \eqref{eq_proof_1} are preserved and we also have
\begin{align}\label{eq_proof_2}
 \inf\limits_{j,k\in\Z, j\neq k}|y_{j}-y_{k}|\geq \delta >0,
    \qquad
    \inf\limits_{x\in\R} \tfrac{\#Y\cap [t-r,t+r]}{2r} \geq \Omega+\varepsilon. 
    \end{align}
Since $\|g_N\|_\infty = 1$, Bernstein's inequality \cite[Chap.~IX]{Levin1980} implies that $g_N \to g$ uniformly on compact sets. Thus, for any open and bounded interval $I \subset \mathbb{R}$, 
    \begin{equation}
        \sup\limits_{y\in Y\cap I}|g(y)| \leq \liminf\limits_{n\to\infty}  \sup\limits_{y\in Y_N\cap I}|g_N(y)|=0.
    \end{equation}
    We conclude that $g \equiv 0$ on $Y$.
    On the other hand, \eqref{eq_proof_2} imply that $Y$ is uniformly separated and $D^-(Y)\geq \Omega+\varepsilon$. Hence, $Y$ is a sampling set for $PW_\Omega^\infty(\R)$ \cite[p.~346]{BeurlingCollected1989} and therefore $g$ is the zero function. 
    However, \eqref{eq:g_N_int_bound} implies that $g$ cannot be zero. 
    We have reached a contradiction, and thus \eqref{eq_proof_0} must indeed hold.

\smallskip

\noindent {\bf Step 2:  {Localization}.} 
Let $\Omega':=\Omega+\varepsilon/2$, so that
\begin{align}
\inf\limits_{t\in\R} \tfrac{\#X\cap [t-r,t+r]}{2r} \geq \Omega'+\varepsilon/2.
\end{align}
Let $\eta \in C^\infty(\mathbb{R})$ with $\mathrm{supp}(\eta) \subset [-1/4,1/4]$ and $\int_\R \eta(t)\,dt =1$. For each $x \in \mathbb{R}$ we consider
\begin{align}
g_x(y) \coloneqq f(y) \cdot\hat{\eta}\left(\varepsilon(y-x)\right), \qquad y \in \mathbb{R}.
\end{align}
Then $g_x(x) = f(x)$ and $g_x \in PW^\infty_{\Omega+\varepsilon/2}(\mathbb{R})$. By Step 1, we conclude that
\begin{align}
|f(x)| = |g_x(x)| \leq \beta \sup_{k\in\mathbb{Z}} |f(x_k)| \, |\hat{\eta}(\varepsilon(x_k-x))|,
\end{align}
where $\beta = \beta^\infty(\Omega+\varepsilon/2,\delta, \varepsilon/2,r)$. Hence,
\begin{align}
\int_{\mathbb{R}} |f(x)|^2\,dx
&\leq \beta^2 
\int_{\mathbb{R}}
\sup_{k\in\mathbb{Z}} |f(x_k)|^2 |\hat{\eta}(\varepsilon(x_k-x))|^2\,dx
\\
&\leq \beta^2 
\int_{\mathbb{R}}
\sum_{k\in\mathbb{Z}} |f(x_k)|^2 |\hat{\eta}(\varepsilon(x_k-x))|^2\,dx
\\
&\leq \beta^2 \sum_{k\in\mathbb{Z}} |f(x_k)|^2 
\int_{\mathbb{R}}
|\hat{\eta}(\varepsilon x)|^2\,dx
\\
&\leq \frac{\beta^2}{\varepsilon} \|\eta\|^2_{L^2(\mathbb{R})} \sum_{k\in\mathbb{Z}} |f(x_k)|^2,
\end{align}
and we can take $A(\Omega, \delta,\varepsilon,r) = (\beta^\infty(\Omega+\varepsilon/2,\delta, \varepsilon/2,r))^{-2} \frac{\varepsilon}{\|\eta\|_2^2}$. (The localization argument used here goes back to Beurling; the optimized version presented here is from \cite[Sec.~2]{OlevskiiUlanovskii2012}.)
\qed

\end{document}